\documentclass[11pt,oneside]{amsart}
\usepackage[english]{babel}
\usepackage{fullpage}
\usepackage{amssymb,pstricks,amscd,epsfig}
\usepackage{graphicx}

\usepackage{amsmath, amsthm, amssymb}
\usepackage{pdfsync}
\usepackage[latin1]{inputenc}
\usepackage{stmaryrd}
\usepackage{color}
\usepackage{setspace}
\usepackage[matrix,arrow,tips,curve]{xy}

\numberwithin{equation}{section}

\makeatletter
 \renewcommand\section{\@startsection {section}{1}{\z@}%
     {-4.5ex \@plus -1ex \@minus -.2ex}%
     {2.3ex \@plus.8ex}%
    {\centering\scshape}}

\newcommand{\Z}{\mathbb{Z}}
\newcommand{\Q}{\mathbb{Q}}

\newcommand{\C}{\mathbb{C}}
\newcommand{\PP}{\mathbb{P}}
\renewcommand{\P}{\mathbb{P}}

\newcommand{\CC}{\mathcal C}

\renewcommand{\H}{\mathcal H}
\renewcommand{\O}{\mathcal O}

\newcommand{\E}{\mathcal E }
\newcommand{\F}{\mathcal F }

\newcommand{\dual}{\vee}
\newcommand{\be}{\begin{equation}}
\newcommand{\ee}{\end{equation}}
\newcommand{\tto}{\dashrightarrow}
\newcommand{\wt}{\widetilde}
\newcommand{\wh}{\widehat}
\newcommand{\ov}{\overline}
\newcommand{\Sb}{\overline{S}}
\newcommand{\St}{\widetilde{S}}
\newcommand{\Db}{\overline{D}}

\newcommand{\Sh}{\widehat{S}}
\newcommand{\rh}{\widehat{\rho}}
\newcommand{\eh}{\widehat{\eta}}
\newcommand{\Dh}{\widehat{D}}
\newcommand{\rhs}{\widehat{\rho}^{\, *}}

\newcommand{\D}{\mathcal{D}}
\newcommand{\mc}{\mathcal}

\newcommand{ \ff} { \frac }

\def\be{\begin{equation}}
\def\ee{\end{equation}}

\newtheorem{theorem}{Theorem}[section]
\newtheorem{prop}[theorem]{Proposition}
\newtheorem{lem}[theorem]{Lemma}
\newtheorem{defin}[theorem]{Definition}
\newtheorem{rem}[theorem]{Remark}
\newtheorem{notation}[theorem]{Notation}

\newtheorem{cor}[theorem]{Corollary}
\newtheorem{example}[theorem]{Example}

\def\id{id}
\def\codim{\operatorname{codim}}

\def\Hom{\operatorname{Hom}}
\def\Aut{\operatorname{Aut}}
\def\coker{\operatorname{Coker}}
\def\End{\operatorname{End}}
\def\im{\operatorname{Im}}

\newcommand{\Shext}{\mathcal{E}xt}
\newcommand{\Shhom}{\mathcal{H}om}

\def\Ext{\operatorname{Ext}}
\def\Tor{\operatorname{Tor}}

\def\supp{\operatorname{Supp}}

\def\tr{\operatorname{tr}}

\def\Hilb{\operatorname{Hilb}}
\def\mom0{\mu^{-1}(0)}

\def\Pic{\operatorname{Pic}}
\def\Prym{\operatorname{Prym}}
\def\NS{\operatorname{NS}}
\def\Amp{\operatorname{Amp}}
\def\Fix{\operatorname{Fix}}
\def\Jac{\operatorname{Jac}}
\newcommand{\into}{\hookrightarrow}
\newcommand{\norm}[2]{\mc{N}_{#1/#2}}

\begin{document}

\title[Relative Prym varieties]{
Relative Prym varieties associated to the double cover of an Enriques  surface}

\author{E. Arbarello,   G. Sacc\`a and A. Ferretti}

\email{ea@mat.uniroma1.it, gsacca@math.princeton.edu, ferrettiandrea@gmail.com}

\date{March 22, 2013}

\begin{abstract}
Given an Enriques surface $T$,  its universal K3 cover $f: S\to T$, and a genus $g$ linear system $|C|$ on $T$, we construct the relative Prym variety $P_H=\Prym_{v, H}(\D/\CC)$, where $\CC\to |C|$ and $\D\to |f^*C|$ are the  universal families, $v$ is the Mukai vector $(0,[D], 2-2g)$ and $H$ is a polarization on $S$. The relative Prym variety is a $(2g-2)$-dimensional possibly singular variety, whose smooth locus is endowed with a hyperk\"ahler structure. This variety is constructed as the closure of the fixed locus of a symplectic birational involution defined on the moduli space $M_{v,H}(S)$. There is a natural Lagrangian fibration $\eta: P_H \to |C|$,  that makes the regular locus of $P_H$ into an integrable system whose general fiber is a $(g-1)$-dimensional (principally polarized) Prym variety, which in most cases is not the Jacobian of a curve.
We prove that if $|C|$ is a hyperelliptic linear system, then $P_H$ admits a symplectic resolution which is birational to a hyperk\"ahler manifold of K3$^{[g-1]}$-type, while if $|C|$ is not hyperelliptic, then $P_H$ admits no symplectic resolution. We also prove that  any resolution of $P_H$ is simply connected and, when $g$ is odd,  any resolution of $P_H$  has $h^{2,0}$-Hodge number equal to one.

\end{abstract}

\maketitle

\tableofcontents


\section{Introduction}\label{introduction}

One of the many beautiful features of the Beauville-Mukai integrable systems, is that they can be observed and studied from two quite different perspectives. The starting point is a linear system $|D|$ of genus $h$ curves  on a K3 surface $S$. If $\D\to |D|=\PP^h$ is the universal family one can view the Beauville-Mukai integrable system as the relative compactified Jacobian $\pi: J(\D)\to \PP^h$,  whose fibers are compactified Jacobians of -say- degree zero. But the hyperk\"ahler and Lagrangian nature of this fibration is not unveiled until one 
sees it from the second perspective where  $J(\D)$ is taken to be the moduli space $M_{v, H}(S)$ of 
$H$-semistable rank zero sheaves with Mukai vector $v=(0, [D], 1-h)$. From this point of view, another aspect comes
to the forefront:  the choice of a polarization $H$, which is inherent to the notion of semistabilty. Its interplay with the  Mukai vector $v$  introduces a subdivision of the ample cone of $S$ into chambers,
bordered by walls. Varying $H$ in $\Amp(S)$ leaves unchanged  the birational type of $M_{v, H}(S)$. When $H$ lies on a wall the corresponding moduli space is singular while,  moving $H$ away from a wall, into the various   adjacent chambers,  correspond to distinct smooth birational models of the same moduli space. When smooth, these moduli spaces are as nice as one can think of: they are irreducible symplectic manifolds meaning that they are compact, simply connected, and that their $H^{2,0}$ space is one-dimensional, spanned by a non-degenerate (i.e. symplectic)  form.
\vskip 0.1 cm
Our task is to study to which extent a similar picture presents itself when Jacobians are substituted with more general abelian varieties. The first natural abelian varieties that come to mind are Prym varieties and the way we make them appear is to start from an Enriques surface $T$, look at its universal K3 cover $f:S\to T$  
and take on $S$ a linear system $|D|$ which is the pull-back, via $f$, of a linear system $|C|$ on $T$ of genus $g \ge 2$.
For each smooth curve $C_0\in |C|$ we look at the double cover  $f: D_0=f^{-1}(C_0)\to C_0$ 
and we consider the Prym variety
$\Prym(D_0/C_0)$ which is a $(g-1)$-dimensional abelian variety. If $U\subset |C|$ is  the  locus of smooth curves we see, right away, a fibration 
$\mathcal P\longrightarrow U \subset \PP^{g-1}$ in $(g-1)$-dimensional, principally polarized, abelian varieties. If $\iota$ is the involution on $D_0$ induced by the two-sheeted covering,
the Prym variety $\Prym(D_0/C_0)$ can be viewed as the (identity component of the) fixed locus in $J(D_0)$
of the involution $-\iota^*$. The way to compactify  $\mathcal P$ is now laid out: define the involution $-\iota^*$ on $J(\CC)=M_{v, H}(S)$ and take its fixed locus. 

When $H$ is $\iota^*$-invariant, the involution $\iota$ causes no problem. It acts on $S$ and $\iota^*$ acts  on the set of $H$-semistable coherent sheaves  on $S$  supported on curves belonging to $|D|$, preserving all their good properties (rank, $H$-semistability) and especially their first Chern class,  which is the $\iota^*$-invariant class $[D]$.

 For the involution
``$j=-1$'' matters are more complicated. For a smooth curve $D_0\in |D|$ the involution $j$ on $J(D_0)$
is given by $j([F])=[F^\vee]=\H om_{D_0}(F, {\mc O}_{D_0})$. Thus, for
$[F]\in M_{v, H}(S)$, a   natural choice 
is to set $j(F)=\Shext^1_S(F, {\mc O}_S(-D))$ (see Lemma \ref{dual}). The $\Shext^1$ functor behaves nicely in families and therefore $j$ induces  a well defined morphism of the deformation functor into itself, the Mukai vector is preserved by $j$ and,  at least  if the support of $F$ is irreducible, both $F$ and $j(F)$ are $H$-stable, {\it for any}
 polarization $H$. Therefore we  get a birational ``-1'' involution: $j: M_{v, H}(S)\dasharrow M_{v, H}(S)$
which commutes with $\iota^*$. However, 
the only way to have it preserve stability is to choose a polarization which is a multiple of $D$. Taking $H=D$ we  then get
a regular map $j: M_{v, D}(S)\to M_{v, D}(S)$, but now $M_{v, D}(S)$ is singular. 

To define the relative Prym variety, we consider  the birational involution 
$$
\tau=j\circ\iota^*:  M_{v, H}(S)\dasharrow M_{v, H}(S)
$$
 and we look at the Lagrangian fibration $M_{v, H}(S)\to |D|$. Over the locus of irreducible curves, this fibaration has a {\it zero-section} $s$ whose restriction to $|C|\subset |D|$ lands in the fixed locus of $\tau$. We then
 define the relative Prym variety $P_H=\Prym_{v, H}(\D/\CC)$ to be the irreducible component of the closure of fixed locus of $\tau$ containing the zero section $s$.
Since the birational involutions $j$ and $\iota^*$ are antisymplectic (see Proposition \ref{prop jN}), a Zariski open subset of $P_H$ has a natural symplectic structure.

Taking a slightly longer view we see that, a priori,  there is more leeway in our choices. We could choose another divisor $N$ and define
$j_N(F)=\Shext^1_S(F, {\mc O}_S(N))$. As long as $N$ is $\iota^*$-invariant and $2\, \chi(F)=N\cdot D$,
we still get a birational involution $j_N: M_{v, H}(S)\dasharrow M_{v, H}(S)$
which commutes with $\iota^*$. It is then natural to ask whether it is possible to choose  $H$, $N$, and $v$
in such a way that $j_N$ is  preserves $H$-stability and, at the same time,  $M_{v, H}(S)$ is smooth. In Proposition \ref{never_preserve}
we show that this is not possible. The alternative: regularity of the assignment $[F]\mapsto [\Shext^1_S(F, {\mc O}_S(N))]$ versus smoothness of $M_{v, H}(S)$
is in the nature of this problem. Moreover, the relative Prym varieties that one can construct out of the involution $\tau_N\circ\iota^*:M_{v, H}(S)\dasharrow M_{v, H}(S)$ enjoy, by and large, the same properties of the relative Prym variety we just defined by
taking $N=-D$.  Consequently this will be our choice throughout the present paper.

The next observation is that the geometry of the relative Prym variety strongly depends on whether the linear system $|C|$ on the Enriques surface $T$ is {\it hyperelliptic} or not (cf. Definition \ref{hyp_syst}).

In the non-hyperelliptic case we show that $P_D$ (and hence $P_H$) does not admit a symplectic desingularization. This is seen by 
going to singular point of $P_D$ that is represented by 
a polystable sheaf $F$ which splits into a direct sum $F=F_1\oplus F_2$ of two stable sheaves supported on irreducible and $\iota$-invariant curves $D_1$ and $ D_2$ respectively. We prove that, locally around the point $[F]$,  the relative Prym variety $P_D$ 
is isomorphic to  its tangent cone at $[F]$ and this can be nicely described in terms of the quadratic term of the Kuranishi map (see Proposition \ref{loc_sing}). The result is that, locally around $[F]$, the relative Prym variety $P_D$ looks like the product of a smooth variety times the cone over the degree two Veronese embedding of a projective space $\PP W$, with $\dim W= D_1\cdot D_2=2C_1\cdot C_2$. This cone is $\Q$-factorial and, if $\dim W\geq 3$, it is also terminal.  When $|C|$, and therefore $|D|$, is non-hyperelliptic, we must have $D_1\cdot D_2\geq 4$, and therefore $P_D$ has no symplectic resolution.

In the hyperelliptic case, things go in the opposite direction. In this case the ``$-1$'' involution $j$ can be nicely described in geometrical terms. The key remark is that, in the hyperelliptic case, not only  $\iota^*$ but also  $j$ comes from an  involution defined on $S$.
Indeed the linear system $|D|$ exhibits
$S$ as a two sheeted ramified cover of a rational surface $R\subset \PP^h$. If $\ell$ is the involution associated to this cover, $\ell^*=j$ is the ``$-1$'' involution on each $J(D)$, when $D$ is smooth. The composition $k=\ell \circ \iota$ is a symplectic involution on $S$ and $k^*$ coincides with $\tau$ as birational maps. 
Taking $S/\langle k \rangle$ and resolving the eight singular points yields a K3 surface $\Sh$. We then show that, for any $\iota^*$-invariant polarization $H$, the relative Prym variety $P_{v,H}$ is birational to an appropriate moduli space of sheaves on $\Sh$ and is therefore of $K3^{[g-1]}$-type.
We also show that, choosing $H$ appropriately, $P_{v,H}$ is a symplectic resolution of $P_{v,D}$.

In Sections \ref {fund_group} and \ref{hilb} we prove that the relative Prym variety $P$ shares two fundamental properties with  the moduli spaces $M_{v, H}(S)$. 
The first result is that any resolution of $P$ is simply connected. To show this we prove that the homology of the  fibers of the Prym fibration $P\to \PP^{g-1}$
is generated by vanishing cycles.
The second result  is that, at least when when the genus $g$ is odd, the $(2,0)$-Hodge number of any desingularization of $P$
is equal to one.

Summarizing, the main results that we prove are the following.

\begin{theorem}\label{main} Let $T$ be a general Enriques surface and $f:S\to T$ its K3 double cover. Let $C\subset T$ a primitive curve of  genus $g\geq 2$.
Set $|D|=|f^{-1}C|$. Set $v=(0, [D], 2-2g)$
an let $P_H=\Prym_{v, H}(\D/\CC)\to |C|=\PP^{g-1}$ be the relative Prym variety. Then

i) If $|C|$ is hyperelliptic  $P_H$ is birational to a hyperk\"ahler manifold of K3 $^{[g-1]}$-type.

ii) If $|C|$ is not  hyperelliptic  $P_H$ admits no symplectic resolution.

\end{theorem}
\begin{theorem}\label{main_top}With the same notation as in Theorem \ref{main}, any resolution of the singularities of $P_H$ is simply connected and, when $g$ is odd (or $|D|$ is hyperelliptic) has $(2,0)$-Hodge number equal to 1.
\end{theorem}

In Section \ref{discriminant}, we compute the degree of the discriminant of the Prym-Lagrangian fibration $P\to |C|$. 
In the case of a compact hyperk\"ahler manifold this degree has been linked by Hitchin and Sawon to the Rozanski-Witten invariant. Comparing the degree of the discriminant of the Prym fibration with the degrees of the discriminants in hyperk\"ahler cases, we highlight a curious numerological phenomenon arising therein.

\subsection*{Acknowledgments} In our work we benefited from the ideas contained in the articles
 by Markushevich  and Tikhomirov,  and  Kaledin, Lehn and Sorger.
During the preparation of this manuscript we had very useful conversations with a number of  people and notably with  D. Huybrechts, J. Koll\'ar, M. Lehn, E. Markman, D. Markushevich, K. O'Grady, E. Sernesi, V. Shende, J. Sawon, C.Voisin. We thank them them all heartedly.
The second named author whishes to express  her gratitude to her advisor G. Tian for his continuous support, and to M. Lehn and J. Sawon for their strong encouragement in an early stage of this project.


\section{Notation}\label{notation}

Let $T$ be an Enriques surface, and let
\be
f: S \to T,\label{univ_cover}
\ee
be the universal cover of $T$. It is well known that $S$ is a K3 surface. We denote by
\be \label{iota}
\iota: S \to S,
\ee
the covering involution. The involution $\iota$ acts as $-1$ on the space $H^0(S, \omega_S)$ of global sections  of the canonical bundle of $S$, i.e., $\iota$ is an anti-symplectic involution. By a result of Namikawa (Proposition (2.3) of \cite{Namikawa1}) the invariant subspace of the involution $\iota^*$ acting on the N\' eron-Severi group\footnote{We follow Definition 1.1.13 of \cite{Lazarsfeld1} and for us the N\'eron-Severi group of a smooth projective variety $X$ is the group of line bundles on $X$ modulo numerical equivalence.} $\NS(S)$ is equal to $f^*(\NS(T))$. Since the pullback
\[
f^*: \NS(T) \to \NS(S),
\]
is injective, we deduce that $f^*\NS(T)$ is a rank $10$ primitive sub-lattice of $\NS(S)$. In particular, the Picard number of $S$ is greater or equal than $10$.
It is well known (Proposition (5.6) of \cite{Namikawa1}) that if $T$ is general in moduli, then
\be \label{Neron-Severi}
\NS(S)\cong f^* \NS(T).
\ee
From now on, when we say that $T$ {\it is general}, we mean that (\ref{Neron-Severi}) holds.
If $T$ is general, then 
\[
\NS(S)\cong \NS(T)(2)\,.\footnote{Given a lattice $\Lambda$ and a non-zero integer $m$, we denote by $\Lambda(m)$ the lattice obtained from $\Lambda$ by multiplying its non-degenerate integral bilinear form by $m$.}
\]
Consequently, the square of any class in $\NS(S)$ is divisible by $4$. In particular, $S$ and $T$ do not contain any algebraic $-2$ classes, i.e., they do not contain any smooth rational curve.
A surface that does not contain any smooth rational curve is called {\it unodal}.

In this paper $C$ will denote a curve on $T$, and we will set
\[
D:=f^{-1}(C).
\]
By abuse of notation, we denote by
\be
\iota: D \to D,
\ee
the induced covering involution.
For any sheaf $F$ on $T$, we set
\be \label{notation F'}
F':=F\otimes \omega_T.
\ee
Then $f^*F\cong f^*F'$. If $C$ is a curve with $C^2 \ge 0$, we usually denote by $C'$ a section of ${\mc O}(C)'$.

Now suppose that $C\subset T$ is an irreducible curve of genus $g$. If $g\ge 2$, it follows from the Hodge index theorem that $D$ is connected and thus $f: D \to C$ is a non trivial \'etale double cover. In particular,
\[
h:=g(D)=2g-1.
\]
 
We say that a curve on a surface is {\it primitive}, if its class is primitive in the N\'eron-Severi group.

If $g\ge 2$, or if $C$ is a primitive elliptic curve, then
\[
\dim |C|=\ff{C^2}{2}=g-1,
\]
while if $C=2C_0$, with $C_0$ a primitive elliptic curve, then $|C|$ is an elliptic pencil with two multiple fibers. 
As for the K3 surface $S$, we have
\[
\dim |L| =\ff{L^2}{2}+1=h.
\]
for every line bundle $L$ on $S$ with $L^2=2h-2\ge 0$.

Consider the covering $f: S \to T$,  a curve $C\subset T$ of genus $g\ge 2$ and  the induced covering $D \to C$. Observe that $\iota^*$ acts on $|D|$, and that the two invariant subspaces are  the $(g-1)$-dimensional spaces
\[
f^*|C|, \,\,\,\, \text{and} \,\,\,\, f^*|C'|.
\]
In the sequel we will drop the symbol $f^*$ and consider $|C|$ and $|C'|$ as subspaces of $|D|$.

In the Appendix we collect a number of facts about linear systems on K3 and Enriques surfaces that are needed in our study.


\section{The relative Prym variety}\label{rel_prym}

\subsection{Pure sheaves of dimension one}\label{pure_sh_dim_one}

The most natural way to compactify the Jacobian variety of an irreducible curve is to consider the moduli space of rank one torsion free sheaves. On reducible curves, however, there is no canonical moduli space to take and, in order to compactify the Jacobian, one has to choose a polarization. For example if the curve is reduced, one can just give a positive integer for every irreducible component of the curve. Different components might appear for different polarizations, and hence the resulting moduli spaces depend on such choice. For (possibly reducible) curves with only nodal singularities this was done by Oda and Seshadri in their fundamental paper \cite{Oda-Seshadri79}. For a survey see \cite{Alexeev04}.

More generally, Simpson \cite{Simpson} showed that it is possible to consider moduli spaces of semi-stable pure sheaves on any polarized smooth projective variety.  Given a sheaf $F$ on a variety $X$, for every $0\le d' \le d$, we denote by $T_{d'}(F) \subset F$ the maximal sub sheaf of $F$ of dimension $d'$.
A sheaf $F$ on  $X$ is said to be {\it pure of dimension $d$} if all the associated prime of $F$ have dimension $d$. Equivalently, if all non-zero sub sheaves of $F$ have support that is of dimension $d$. Compactified Jacobians of arbitrary curves lying on smooth projective surfaces are thus (closed subsets of) special cases of these Simpson moduli spaces.

In this subsection we let $(X, H)$ denote a smooth projective surface together with a polarization $H$ and $F$ a pure sheaf of dimension one on $X$. 

We start by recalling a few important features of pure dimension one sheaves on a smooth surface (cf. \cite{Huybrechts-Lehn} for a more detailed discussion).
By Proposition 1.1.10 of \cite{Huybrechts-Lehn}, the sheaf $F$ has homological dimension one, and hence has a two step locally free resolution.

\be\label{loc_free}
0 \to L_1 \stackrel{a}{\longrightarrow} L_0 \to F \to 0.
\ee
The \emph{determinantal support} of $F$, denoted by $\supp_{\det}(F)$, is the curve  in $X$ defined by the vanishing of the determinant of $a: L_1 \to L_0$. We have,
\[
[\supp_{\det}(F)]=c_1(F),
\]
where $[\Gamma]$ denotes the class in $H^2(X)$ of a curve $\Gamma\subset X$.
The \emph{scheme theoretic support} of $F$, is the sub-scheme whose sheaf of ideals is the kernel of the natural morphism
$
{\mc O}_X \to \mc{E}nd(F).
$
Unless explicitly specified we will use the word {\it support}
to refer to the determinantal support.
Observe that the determinantal support behaves well in families, whereas the scheme theoretic one does not.

In this case, the notion of Gieseker stability in terms of the reduced Hilbert polynomial translates into the following definition

\begin{defin} The slope of $F$ with respect to $H$ is the rational number,
\[
\mu_H(F):=\ff{\chi(F)}{c_1(F)\cdot H}.
\]
A sheaf $F$ is $H$-(semi)-stable if and only if (we follow Notation 1.2.5 of \cite{Huybrechts-Lehn}) it is pure of dimensions one and if, for every quotient sheaf $F\to E$ that is pure of dimension one,
\[
\mu_H(F) (\le) \mu_H(E).
\]
\end{defin}

Clearly, stability can be formulated also in terms of saturated subsheaves of $F$, by reversing the inequality in the definition.

By definition,  any pure sheaf of dimension one and rank one supported on an integral curve is automatically stable with respect to any polarization. More interesting phenomena arise when the support is non integral.

Let $\Gamma\subset X$ be the scheme theoretic support of $F$, and let $\Gamma' \subset \Gamma$ be any sub curve. The restriction $F_{|\Gamma'}:=F \otimes {\mc O}_{\Gamma'}$ is not necessarily pure of dimension one, so We set
\be
F_{\Gamma'}:=F_{|\Gamma'}/T_0(F_{|\Gamma'}).
\ee
so that $F_{\Gamma'}$ is pure of dimension one.

\begin{lem}  \label{stability sub curves}
Let $F$ be a pure sheaf of dimension one and rank one on a reduced curve $\Gamma \subset X$. Then $F$ is $H$-(semi)-stable if and only if for every sub curve $\Gamma' \subset \Gamma$
\[
\mu_H(F) (\le) \mu_{H}(F_{\Gamma'}).
\]
\end{lem}
\begin{proof}
Let $\alpha: F \to G$ be a quotient sub sheaf of pure dimension one, and let $\Gamma'$ be the support of $G$. Then $\alpha$ factor through the natural morphism $F \to F_{\Gamma'}$ and the induced morphism $F_{\Gamma'} \to G$ is an isomorphism.
\end{proof}

\begin{notation}\label{Fi}Let $F$ be a pure sheaf of dimension one on a smooth surface $X$, and let $\Gamma=\Gamma_1+\Gamma_2$ be a decomposition of the support of $F$ in two reduced curves that have no common components. For $i,j=1,2$ and $i\neq j$ we set
\be \label{Fi3}
F^{ \Gamma_j}=\ker[F_{|\Gamma_i} \to F_{\Gamma_i}],
\ee
When no confusion is possible we will use the shorthand notation
\be\label{Fi2}
F_i=F_{\Gamma_i}\,,\qquad F_j=F^{ \Gamma_j}
\ee

\end{notation}

We will  need  the following lemma.

\begin{lem} \label{Fj}
Let $F$ be a pure sheaf of dimension one on a smooth surface $X$, and let $\Gamma=\Gamma_1+\Gamma_2$ be a decomposition of the support of $F$ in two reduced curves that have no common components. Assume that $\Gamma_1$ and $\Gamma_2$ meet transversally. Let $\Lambda_F$ be the subset of $\Gamma_1\cap \Gamma_2$ where $F$ is locally free and set $\Delta_F=\underset{p\in \Lambda_F}\sum p$.
Then, for $j=1,2$
\[
F^{\Gamma_j}\cong F_{\Gamma_j}\otimes {\mc O}_{\Gamma_j}(-\Delta_F),
\]
\end{lem}
\begin{proof} For $i,j=1,2$, $i\neq j$, 
look at the exact sequence
\be \label{structure sheaf sequence}
0 \to {\mc O}_{\Gamma_j}(-\Gamma_i)\to {\mc O}_\Gamma \to {\mc O}_{\Gamma_i} \to 0.
\ee

Tensoring by $F$ we get
\[
\Tor^1_{{\mc O}_\Gamma}(F,{\mc O}_{\Gamma_i}) \to F_{|\Gamma_j}(-\Gamma_i) \stackrel{t}{\to} F \to F_{|\Gamma_i} \to 0, 
\]
where the sheaf $\Tor^1_{{\mc O}_\Gamma}(F,{\mc O}_{\Gamma_i})$ is supported on 
\[
T_F:=\Gamma_1\cap \Gamma_2 \setminus \Lambda_F.
\]
For $p\in T_F$,
\[
F_{|\Gamma_i,p}\cong F_{\Gamma_i,p}\oplus \C_p, \,\,\,\, \text{and} \,\,\,\,
F_{|\Gamma_j}(-\Gamma_i)_p\cong F_{\Gamma_j}(-\Gamma_i)_p\oplus \C_p.
\]
Since $F$ is pure, the map $t$ factors through a generically injective (and thus injective) map $F_{\Gamma_j}(-\Gamma_i) \stackrel{s}{\to} F$ which fits into the following exact sequence
\[
0 \to  F_{\Gamma_j}(-\Gamma_i) \stackrel{s}{\to} F \to F_{|\Gamma_i} \to 0.
\]
Consider the following commutative diagram
\[
\xymatrix{
0 \ar[r] &   F_{\Gamma_j}(-\Gamma_i) \ar[r]  \ar[d]_\gamma & F \ar[r] \ar@{=}[d]  & F_{|\Gamma_i} \ar[r] \ar@{->>}[d]^\beta & 0\\
0 \ar[r] & \ker(\alpha) \ar[r] & F\ar[r]^\alpha &  F_{\Gamma_i}  \ar[r] & 0
}
\]
where $\alpha: F \to F_{\Gamma_i} $ is the composition of the restriction $F \to F_{|\Gamma_i}$ with the natural morphism $\beta$.
We know that $\beta$ is an isomorphism at a point $p$ if and only if $F$ is locally free at $p$. Moreover, $\gamma$ is injective and  
\[
\ker(\beta)\cong \oplus_{p\in T_F } \C_p.
\]
It follows that $\coker (\gamma) \cong \oplus_{p\in T_F } \C_p$, and thus
\[
\ker(\alpha)\cong F_{\Gamma_j}(-\Gamma_i)\left(\sum_{p\in T_F }p\right) = F_{\Gamma_j}(-\Delta_F).
 \]
\end{proof}

\vskip 0.5 cm 
\subsection{Relative compactified Jacobians}

Let $(S, H)$ be a polarized K3 surface, and let
\[
v =(0, [D], \chi) \in H^*(S, \Z),
\]
be a Mukai vector, where 
\[
\chi=d-h+1\,,\quad h=g(D)
\]
Following Le Potier \cite{LePotier} and Simpson \cite{Simpson}, we consider the moduli space $M_{v, H}(S)$ of $H$-semi stable sheaves of pure dimension one with $c_1(F)=[D]$ and $\chi(F)=\chi$. 
The moduli space $M_{v, H}(S)$  is a $2h$-dimensional projective variety and by \cite{Mukai} and \cite{Artamkin88} the smooth locus contains the locus $M_{v, H}^s(S)$ of $H$-stable sheaves.

When no confusion is possible we will simply write:
\be\label{simplifyMvH}
M_{v, H}=M_{v, H}(S)
\ee

Let $[F] \in M_{v, H}$ be a point corresponding to an $H$-stable sheaf. By deformation theory, the tangent space to $M_{v, H}$ at the point $[F]$ is canonically identified with $\Ext^1(F, F)$. Mukai showed in \cite{Mukai} that there is a holomorphic symplectic form on $M_{v, H}^s$. On the tangent space $T_{[F]} M_{v, H}^s$ this symplectic form is given by the the cup product
\be\label{cup_ext}
\aligned
\sigma_M: \Ext^1(F, F) \times \Ext^1(F, F) \longrightarrow \Ext^2(F, F)& \stackrel{tr}{\cong} H^2(S, {\mc O}_S)\\
&= H^0(S,\omega_S)^\vee\stackrel{\sigma}{\cong} 
\C,
\endaligned
\ee
Following Le Potier \cite{LePotier} we can define the support morphism
\be
\begin{aligned}\label{lagrang}
\pi: M_{v, H} & \longrightarrow |D|\cong \P^h,\\
F & \longmapsto \supp_{\det}(F).
\end{aligned}
\ee
The fiber of $\pi$ over a point corresponding to a smooth curve $D_0$ is nothing but $\Jac^d(D_0)$, the degree $d$ Jacobian of $D_0$. 
If the curve $D_0$ is not smooth but integral, the fiber of $\pi$ over the point $[D_0]$ is the compactified Jacobian $\overline \Jac^d(D_0)$ whose points represent isomorphism classes of rank one degree $d$ torsion free sheaves on $D_0$.
More generally, if $D_0$ is reduced but possibly reducible, the fiber is the compactified jacobian $\overline \Jac^d_H(D_0)$ parametrizing $\mc S$-equivalence classes of $H$-semistable rank one torsion free sheaves on $D_0$.

For this reason the moduli space $M_{v, H}$ is sometimes denoted with the symbol 
\[
\Jac^d_H(|D|)=M_{v, H}
\]
or else with the symbol 
\[
\Jac^d_H(\D)=M_{v, H}
\]
 where $\D\to |D|$ is the universal family.

 We recall that a polarization $H$ is called \emph{$v$-generic} if every $H$-semi stable sheaf of Mukai vector $v$ is automatically $H$-stable.

Yoshioka \cite{Yoshioka} proved that if $v$ is primitive and if $\chi\neq 0$, the locus of $[H] \in \Amp(S)$ that are not $v$-generic is a finite union if hyperplanes which are called the {\it walls associated to} $v$.
These walls are described as follows. Let $[F] \in M_{v, H}$ and let $D$ be its support. For any sub curve $\Gamma \subset D$,  and any quotient sheaf $F \to G$ with $\supp(G)=\Gamma$ and $\mu_H(G)=\mu_H(F)$, there is a wall containing $[H]$ defined in $\Amp(S)$ by the equation
\[
(\chi(G) D -\chi \, \Gamma )\cdot x=0.
\]
By definition, when $H$ is $v$-generic, the moduli space $M_{v, H}$ is smooth. It is an irreducible symplectic manifold of K3$^{[h]}$-type.

A simple example of a non $v$-generic polarization is the following.

\begin{example} \label{D not v generic}
Let $\chi=-h+1$ or, equivalently, $d=0$. Suppose that $D$ decomposes into the sum $D=D_1+D_2$ of two integral divisors, with even intersection numbers $D\cdot D_i$, $i=1,2$. Then $D$ is not $v$-generic. In fact,  there exists a sheaf $F=F_1\oplus F_2$, where $F_i$ is a sheaf on $D_i$ with $\chi(F_i)= -\ff{1}{2} D_i \cdot D$.
\end{example}

One of the beautiful features of  the map (\ref{lagrang}) is that it exhibits $M_{v, H}$ (or rather its smooth locus) as a Lagrangian fibration. Over the locus of smooth curves this was proved by Beauville in \cite{Beauville88}, and for this reason this Lagrangian fibration is called the Beauville-Mukai integrable system. If $v$ and $H$ are such that $M_{v,H}$ is smooth (eg. if $v$ is primitive and $H$ is general), then by a  general theorem of Matsushita \cite{Matsushita1} every irreducible component of every fiber is Lagrangian\footnote{Following \cite{Matsushita2000} we say that a subvariety $Y $ of a complex manifold $X$ with a holomorphic symplectic form $\sigma$ is called Lagrangian, if $\dim Y=\dim X/2$ and if there is a resolution $r: \wt Y  \to Y$ of the singularities of $Y$ such that $r^*\sigma_{|Y}$ is identically zero.}. In particular, $\pi: M_{v,H} \to |D|$ is \emph{equidimensional}.\\
If $d=0$, then $\pi$ has a rational section,
\be \label{rational section}
s: |D| \dasharrow \Jac^0_H(\D),
\ee
which is defined on an open subset containing the locus of integral curves. Indeed, since any pure sheaf of rank one on an integral curve is stable with respect to any polarization, one can define the section $s$ by  assigning to an integral curve $\Gamma \in |D|$ its structure sheaf.

\subsection{The relative Prym variety}\label{rel_prym_sub}

We recall  the classical definition of Prym variety.
Let $C$ be a smooth genus $g\ge2$ curve, and let
\[
f: D \to C,
\]
be an \'etale double cover. Then $D$ is a smooth curve of genus $h=2g-1$. 
As usual, let $\iota: D \to D$ be the covering involution. Then $\iota^*$  acts on the Jacobian variety $\Jac^0(D)$, and the {\it Prym variety  of $D$ over $C$} is  defined by
\be \label{definition Prym}
\Prym(D/C):=\Fix^\circ(-\iota^*)=\ker(\id+\iota^*)^\circ \subset \Jac^0(D),
\ee
where the superscript ${}^\circ$ stands for the identity component. Following Mumford \cite{MumfordPryms}, one can also define the Prym variety as the identity component of the kernel of norm map,
\[
\begin{aligned}
\operatorname{Nm}: \Jac^0(D) &\longrightarrow \Jac^0(D),\\
L=\mc O_D\left(\sum a_i x_i\right) &\longmapsto \det (f_* L)=\mc O_C\left(\sum a_i f(x_i)\right)
\end{aligned}
\]
or, equivalently,  as the image of $(1-\iota^*): \Jac^0(D) \to \Jac^0(C)$.
Since
\[
f^* (\operatorname{Nm}(L))=(1+\iota^*)L,
\]
and $\ker (\operatorname{Nm})$  and $\ker f^*$ have two connected components (cf. \cite{MumfordPryms} \S 6), $\ker(1+\iota^*)$
has four connected components. In other words,
\be \label{ker norm}
\ker(1+\iota^*)\cong \operatorname{Nm}^{-1}(0)\coprod \operatorname{Nm}^{-1}(\eta),
\ee
where $\eta \in \Pic^0(C)$ is the $2$-torsion line bundle defining the double cover $D \to C$.
In \cite{MumfordTheta} Mumford shows, among other things, that the decomposition in connected components of $\ker(\operatorname{Nm})$ is given by
\be \label{components ker norm}
\ker (\operatorname{Nm})=(1-\iota^*)(\Jac^0(D))\coprod (1-\iota^*)(\Jac^1(D)),
\ee
so that $\Prym(D/C) \cong (1-\iota^*)(\Jac^0(D))$.
The Prym variety is a $(g-1)$-dimensional principally polarized abelian variety(cf \cite{MumfordPryms}).
Going back to our situation, we consider a general Enriques surface $T$ and its universal cover 
\be\label{univ_cov}
f: S\to T
\ee
We fix
 a curve $C\subset T$ of  genus $g \ge 2$ and  we set $D=f^{-1}{(}C)$, so that  $\dim |C|= g-1$ and  $\dim |D|=h$.  Set
\[
W:=f^*|C|\subset |D|.
\]
Let 
$\CC\to |C|$ and $ \D\to W\cong|C|$ be the universal families relative to the two linear systems.  Consider the relative cover
 \[
\xymatrix{\D\ar[rr]^F\ar[dr]&&\CC\ar[dl]\\
&|C|
}
\]
Our aim is to preform the Prym construction  for the relative cover $F$.
Of course, we could also do the relative construction staring with the linear system   $W':=f^*|C'| $ and everything we will say for $W$ works  for $W'$ as well.
As in the case of the double cover of  a fixed  curve, we would like to define the relative Prym variety as the fixed locus of an involution defined on the relative jacobian $\Jac^0_H(\D)=M_{v,H}$ where, {\it throughout}
\be\label{choice_v}
v=(0, [D], -h+1)
\ee

 and $H$ is a suitable polarization. Moreover, this involution should be the composition of a relative version of $\iota^*$ with a relative version of  ``$-1$", the two involutions should commute, and they should  be anti-symplectic so their composition would be a symplectic involution. The only part in this construction which is not straightforward, when not downright impossible,  is the construction of the involution ``$-1$".

As we said, the desired involution on $M_{v, H}$ should be the composition of two commuting anti-symplectic involutions. Let us start by describing the first one.

\begin{lem} \label{lem iota star}
Let $T$ be an Enriques surface. Let $S$ be the covering K3 surface. Let $H$ be a polarization on $S$, and let $v=(0,[D], \chi)$ where $D=f^*(C)$. There is a birational involution
\[
\begin{aligned}
\iota^*: M_{v, H} & \dasharrow M_{v, H},\\
F &\mapsto \iota^* F,
\end{aligned}
\]
This birational involution is anti-symplectic and the projection $\pi: M_{v, H}\to |D|$ is $\iota^*$-equivariant. If $H$ is $\iota^*$-invariant, then $\iota^*$ is a regular morphism.
\end{lem}
\begin{proof} Since the general point $[F] \in M_{v, H}$ is supported on an irreducible curve, it is stable with respect to any polarization. This shows that $\iota^*$ is a birational involution, and it is clear that if $H$ is $\iota^*$-invariant, then the involution is biregular. The $\iota^*$-equivariance of $\pi$ is obvious.
The symplectic form on $M_{v, H}$ is given by (\ref{cup_ext}) and there all morphisms  are intrinsic except for the identification $H^0(S, \omega_S)^{\dual}\cong \C$, which is dual to the isomorphism $H^0(S, \omega_S) \cong \C \sigma$. As $\iota$ is an anti-symplectic involution, $\iota^{*}(\sigma) = -\sigma$, and the symplectic form on $M_v(S)$ changes sign under $\iota^{*}$. 
\end{proof}

Every component of the fixed locus of $\iota^*$ is a Lagrangian subvariety of $M_{v,H}$ with trivial canonical bundle,  and their geometry is studied by the second named author in \cite{Sacca12}.

The second involution is more involved. The basic tool one uses to define it is given by the following lemma.

\begin{lem}\label{dual}
Let $F$ be a pure sheaf of dimension one on a K3 surface $S$, and let $\Gamma$ be the support of $F$.
Then,
\[
\Shext^1_S(F, {\mc O}_S(-\Gamma))\cong \Shhom_\Gamma(F, {\mc O}_\Gamma).
\]
\end{lem}
\begin{proof}
Consider the short exact sequence $0 \to {\mc O}_S(-\Gamma)  \to {\mc O}_S\to {\mc O}_\Gamma \to 0$. Applying $\Shhom_S(F, \cdot)$ we get,
\[
0 \to \Shhom_S(F,{\mc O}_\Gamma) \to \Shext^1_S(F, {\mc O}_S(-\Gamma) ) \stackrel{u}{\longrightarrow} \Shext^1_S(F, {\mc O}_S).
\]
Notice, however, that the map $u$ is induced by multiplication by the section defining $\Gamma$. Thus $u=0$ and  $\Shhom_\Gamma(F, {\mc O}_\Gamma) \cong \Shext_S^1(F, {\mc O}_\Gamma(-\Gamma))$.
\end{proof}

We set
$$
j(F):=\Shext^1_S(F, \O(-D))
$$
A pure sheaf of dimension one on a surface is reflexive (Proposition 1.1.10 of \cite{Huybrechts-Lehn}), so that \[
j^2 (F)\cong F
\]
The idea is that the assignment $j$ should be the relative version of the involution 
\[
-1:\Jac^0{(}C)\to \Jac^0{(}C)
\]
so that $j\circ\iota^*$ is   the relative version of the involution 
$-\iota^*:\Jac^0{(}C)\to \Jac^0{(}C)$ whose fixed locus is the Prym variety. From now on we set
\be\label{tau}
\tau:=j\circ\iota^*\,,\qquad \text{i.e.}\quad\tau(F)=\Shext^1(\iota^*F,-D)
\ee

\begin{lem} \label{base change ext}
Let $\F$ be a flat family of pure sheaves of dimension one on $S$ 
parametrized by a scheme $B$. If $p: S \times B \to S$ denotes the natural projection, then 
$\Shext^1(\F, p^*{\mc O}_S)$
is a flat family of pure dimension one sheaves on $S$ parametrized by $B$, and for every $b \in B$, there is an isomorphism 
$ \Shext^1(\F, {\mc O}_{S\times B})_b\cong \Shext^1(\F_b, {\mc O}_{S})$.
\end{lem}

\begin{proof} The lemma follows from \cite{Altman-Kleiman_Compact} in the following way.
First,  for every $b \in B$ we have
$\Shext_S^i(\F_b, {\mc O}_S)=0$, for $i=0,2$. Then
point (ii) of Theorem (1.10) in loc. cit. implies that $\Shext^1(\F, p^*{\mc O}_S)$ is flat and that the base change map $ \Shext^1(\F, {\mc O}_{S\times B})_b\to \Shext^1(\F_b, {\mc O}_{S})$ is an isomorphism.
\end{proof}

Let $v$ be as in (\ref{choice_v}). Using Lemma \ref{base change ext}   the assignment: $F\longmapsto \tau(F)$
yields a well defined involution 
\be\label{tau_N_moduli}
\tau: M_{v, H}\longrightarrow M_{v, H}
\ee

if the following conditions are satisfied:

\be
\aligned\label{cond_abc}
a)&\quad \tau^2(F)\cong F\\
b)&\quad v(\tau(F))=v(F)\\
 c)&\quad\text {\it  $F$  is  $H$-semistable if and only if $\tau(F)$  is $H$-semistable.}
 \endaligned
\ee

\vskip 0.3 cm

a) The condition $\tau^2(F)\cong F$ is equivalent to the condition that $j$ and $\iota^*$ commute and this happens since $D$ is $\iota^*$-invariant.

\vskip 0.3 cm

b) Here we demand
\be\label{mukai_j}
c_1(F)=c_1(j(F))=c_1(\Shext^1_S(F, \O(-D)))\,,\qquad \chi(F)=\chi(j(F))
\ee
 The  first condition is always satisfied.
Indeed, since tensoring by a line bundle does not change the first Chern class of a sheaf supported on a proper subscheme, we may as well replace $\O_S(-D)$ with $\O_S$. Consider, as in (\ref{loc_free}),  a locally free resolution of $F$  so that $c_1(F)$ is the class of the curve defined by the equation $\det a=0$. Dualizing we get:
$
0 \to L_1^\vee  \stackrel{a^\vee}{\longrightarrow} L_0^\vee \to \Shext^1_S(F, {\mc O}_S) \to 0.
$
Since $(\det a=0)$ and $(\det a^\vee=0)$ define the same subscheme of $S$,  the first equality in (\ref{mukai_j_N}) follows. As far as the second equality is concerned, let us compute Hilbert polynomials.
For $m>>0$ and for any pure sheaf of dimension one, we have
\be\label{hilb_j}
\begin{aligned}
p_{j(F)}(m, H)=&\chi(\Shext^1_S(F, {\mc O}(-D+mH))=\dim H^0(\Shext^1_S(F, {\mc O}(-D+mH))\\
=&\dim \Ext^1(F\otimes {\mc O}(D-mH), {\mc O}_S)\\
=&\dim H^1(F\otimes {\mc O}(D) \otimes {\mc O}(-mH))\\
=&-\chi(F\otimes {\mc O}(D) \otimes {\mc O}(-mH))\\=& -p_{F\otimes \O(D)}(-m, H).
\end{aligned}
\ee
In particular we get
\be \label{chi j F1}
\chi(j(F))=-\chi(F\otimes \O(D)  )=-\chi(F)-c_1(F)\cdot D
\ee
When $[F]\in M_{v,H}$ we have
 $c_1(F)=[D]$ and thus $\chi(j(F))=\chi(F)=-h+1$.
\vskip 0.3 cm

c) This is the most delicate and interesting point. Here the 
polarization $H$ comes to the forefront. The question is: for which choice of  $H$ the functor $\tau$ preserves $H$-semistability?
If we only care about the existence
of a birational involution
\be\label{tau_birat}
\aligned
\tau: M_{v, H}&\dashrightarrow M_{v, H}\hskip 3 cm  v=(0,[D],-h+1)\\
[F] &\longmapsto [\tau(F)]
\endaligned
\ee
the question we just raised, is irrelevant, since any pure sheaf of rank one supported on an irreducible curve is automatically stable with respect to any polarization. Hence $\tau$ always exists as a \emph{birational map},  as long as conditions a) and b) are satisfied. 

A second remark is that $\tau$ certainly preserves $D$-stability. For this it 
 suffices to check that $j$ preserves $D$-stability. In fact, on the one hand $j$ establishes a bijection between  pure dimension one  subsheaves of $j(F)$ and  pure dimension one  quotients of $F$ and, on the other it follows from  (\ref{chi j F1}) that,
 given any subsheaf $A$ of $j(F)$, the condition $\mu(A) \geq \mu(j(F))$ is equivalent to the condition
 $\mu(F) \leq \mu(j(A))$.
Thus we have a well defined involution
\be\label{tau_D}
\aligned
\tau: M_{v, D} &\longrightarrow M_{v, D}\hskip 3 cm  v=(0,[D],-h+1)\\
[F]&\longmapsto [\tau(F)]
\endaligned
\ee
The drawback of choosing $D$ as  the polarization, is that $D$ is not $v$-generic (cf Example \ref{D not v generic}) and thus the moduli space $M_{v, D}$ is singular.
This is the price we have to pay in order to have the involution be a regular morphism. If we only care about a birational involution  any choice of $H$ is admissible, in particular, a $v$-generic one.
In the next subsection \ref{choose_pol} we will discuss, in a more general setting, what happens when we vary the polarization $H$.
We now come to the central definition of this section.

\begin{defin} \label{def_rel_prym} Let $T$ be an Enriques surface and $f: S\to T$ its universal cover. Let $C$ be a smooth curve on $T$ of genus $g\geq2$. Let $D=f^{-1}(C{)}$ so that $g(D)=2g-1$. 
 Let $ v=(0,[D],-h+1)$, with $h=g(D)$. Let $H$ be a polarization on $S$ and let $\tau$ be the 
birational involution on $M_{v,H}$ defined by (\ref{tau_birat}).
The  relative Prym variety 
$\Prym_{v, H}(\D/\CC)$ is defined by:
\be\label{prym_N}
\Prym _{v, H}(\D/\CC)=\ov{\operatorname{Fix}^0(\tau)}\subset \operatorname{Jac}^0_H(\mathcal D)=M_{v,H}
\ee
By this we mean that we first look at the fixed locus of the restriction of $\tau$ to an open subset where $\tau$ is a regular morphism, then we consider the closure of the irreducible component of this locus
containing the zero-section.
\end{defin}

When $H=D$, and $\tau$ is regular there is no need to take the closure.
When no confusion is possible we will adopt the shorthand notation:
\be\label{notation_prym}
P_{v,H}=\Prym _{v, H}(\D/\CC)\,.
\ee

The  cases in which $H\neq D$ and $H=D$ are rather 
different in nature. In the first case, at least when $v$ is primitive and $H$ is $v$-generic,
we are taking the closure, in a {\it smooth ambient} space, of a sublocus 
which is defined in a proper open subset and on whose adherence we have no control. In the second case we are taking the fixed locus of a {\it regular  involution defined in a singular space} and, as we will see, this  yields a singular space. 

When $H=D$ we have a good control on the singular locus of $P_{v,D}$. The dimension of this singular locus depends on wether the linear system $|D|$ is hyperelliptic or not (cf. Appendix \ref{appendix} for the relevant definitions).
 
 \begin{prop}\label{codim_sing}Let $P=P_{v,D}$ and let $P_{sing}$ be its singular locus. If $|D|$ is a hyperelliptic linear system then $\codim_P P_{sing}=2$, if  $|D|$ is not hyperelliptic then $\codim_P P_{sing}\geq 4$ .
\end{prop}
\begin{proof} We can stratify $P_{sing}$ by locally closed subvarieties that are isomorphic to open subsets of products of symmetric power of lower dimensional relative Prym varieties. The maximal dimensional strata of $P_{sing}$ are isomorphic to open subsets (or finite quotients of subsets) of relative Prym varieties
of the form $P_{v_1,D}\times P_{v_2,D}$ corresponding to a decomposition  $D=D_1+D_2$. 
Set $h_1=g(D_1)$, $h_2=g(D_2)$. Then 
$h=h_1+h_2+2\nu-1$, where $2\nu =D_1\cdot D_2$. Hence 
\[
\dim(P_{v_1, D}\times P_{v_2, D})=h_1-1+h_2-1=h-1-2\nu .
\]
Since by Corollary \ref{two_curves} $\nu=1$ if and only if $|D|$ is hyperelliptic, the proposition follows.
\end{proof}

The next proposition shows that the smooth locus of the relative Prym  variety $P_{v, H}$ carries a natural symplectic structure.

\begin{prop} \label{prop jN}
The  birational involution (\ref{tau_birat})  $\tau$ is symplectic.
\end{prop}

\begin{proof} By Lemma \ref{lem iota star}, to prove that $\tau$ is symplectic it is enough to prove that
$j$ is anti-symplectic. We argue as follows. Set $M=M_{v, H}$, let $D_0\in |D|$ be a smooth curve, and set $J=\Jac^0(D_0)=\pi^{-1}(D_0)$. Since $\pi$ is a Lagrangian fibration, the isomorphism $T_M \cong\Omega_M^1$ induced by the symplectic form $\sigma_M$, yields an isormophism of short exact sequences,
\[
\xymatrix{
0 \ar[r] & T_J \ar[d]^{\cong} \ar[r] & {T_M}_{|J} \ar[d]^{\cong} \ar[r] & \norm{J}{M} \ar[d]^{\cong} \ar[r] & 0 \\
0 \ar[r] & \norm{J}{M}^{\dual} \ar[r] & {\Omega_M^1}_{|J} \ar[r] & \Omega_J^1 \ar[r] & 0.
}
\]
In particular for a point $x \in J$ we have the isomorphism
\be \label{lagrangian isomorphism}
T_{D_0} |D| \cong \norm{J}{M,x} \cong (T_x J)^{\dual}.
\ee
Since the second isomorphism in (\ref{lagrangian isomorphism}) is given by $\sigma_M$, and since $j^*$ acts as the identity on $T_{D_0} |D|$ and as $-1$ on $T_x J$, we  conclude that $j^*(\sigma_M)=-\sigma_M$. 
\end{proof}

Since $\tau$ respects the fibration $\pi: M_{v,H}\to |D|$, there is a commutative diagram
\be\label{prym_lagran}
\xymatrix{P_{v,H}\ar[d]_{\nu} \ar[r]&M_{v,H}\ar[d]\\
|C|\ar[r]&|D|
}
\ee
where, as usual, we identify $|C|$ with $f^*|C|\subset |D|$.

We sum up the results in the following Theorem

\begin{theorem}
The relative Prym variety $P_{v,H}$ is a $(2g-2)$-dimensional projective variety whose smooth locus $P_{v,H}$ carries a holomorphic two form. This two form is symplectic on a dense open subset and with respect to this symplectic form the morphism
\[
\nu: P_{v,H}  \to |C|,
\]
has a natural structure of Lagrangian fibration.
\end{theorem} 
\begin{proof} The first statement follows directly from the Proposition above. As for the second, notice that 
The map $\nu$ is a {\it Lagrangian fibration} , in the sense that it is such on the open locus of $P_{v,H}$ where the symplectic form is defined. The general fiber of this fibration is a principally polarized abelian variety. Indeed, given any smooth curve $C$ in $|C|$ we have:
\[
\nu_N^{-1}(C)\cong \Prym (D/C),\qquad D=f^{-1}({C})
\]
\end{proof}

We should  point out that if $H=D$ then the form is symplectic on the entire smooth locus of $P_{v,D}$ whereas this is not necessarily true if $H \neq D$ (cf. Example 9.7 of \cite{Sacca13}).

In Sections \ref{singular} and \ref{hyp} we will analyze the singularities of a relative Prym variety, and answer the natural question of whether they admit a symplectic resolution.

\subsection{Prym varieties of singular curves}\label{prym_sing_curv}

Looking at the Lagrangian fibrations (\ref{prym_lagran}),
we now describe the fibers of $\nu$  over points of  $|C|$ corresponding to class of (mildly)  singular curves. 
As usual, for a curve $C \subset T$ in the linear system $|D|$, we set $D=f^{-1}(C)$.
To fix notation, let $m: \wt C \to C$ and $n: \wt D \to D$ be the respective normalizations. Let $\wt f: \wt D \to \wt C$ be the induced double cover and let,
\[
\wt \iota: \wt D \to \wt D,
\]
be the corresponding involution. With this notation there is a commutative diagram
\[
\xymatrix{
\wt D \ar[r]^n \ar[d]_{\wt f} & D \ar[d]^f \\
\wt C \ar[r]^m & C
}
\]

\subsubsection{The irreducible (nodal) case}\label{prym_irr}

Here we consider the case where $C$ and $D=f^{-1}(C)$ are irreducible nodal curves. Notice that in this case, the polarization is irrelevant. Let $c_1, \dots, c_\delta$ be the nodes of $C$ and, and let $p_1, \iota(p_1), \dots, p_\delta, \iota(p_\delta)$,  $i=1,\dots,\delta$, with $f(p_i)=c_i$, be the nodes of $D$. For $i=1,\dots,\delta$ set 
\[
\{x_i, y_i\}=n^{-1}(p_i),  \quad \text{so that } \quad \{\wt \iota x_i, \wt \iota y_i\}=n^{-1}(\iota p_i).
\]

With this notation we have the following proposition.

\begin{prop} \label{components irreducible curves}
Let $C$ and $D=f^{-1}(C)$ be irreducible nodal curves. Then $\Fix(\tau)\subset \overline \Jac^0(D)$ has $4$ connected components, each of which is isomorphic to a rank $\delta$ degeneration of an abelian variety.
\end{prop}

\begin{proof}
To prove the Proposition, we first consider the intersection 
$$
\Fix(\tau)\cap\Jac^0(D)=\ker(1+\iota)
$$
 and prove that it has four connected components. 
Then we prove that these components stay disconnected even after passing to their closure, the last statement of the proposition will be clear from the construction.

For the first step we argue as follows.
Since $D$ and $C$ are irreducible and nodal there are natural short exact sequences,
\[\begin{aligned}
&0 \to  \prod_{i=1}^\delta \C_{p_i}^* \times \C_{\iota p_i}^* \to  \Jac^0(D) \stackrel{n^*}{\rightarrow} \Jac^0(\wt D) \to 0,\\
&0 \to  \prod_{i=1}^\delta \C_{c_i}^* \to  \Jac^0(C) \stackrel{m^*}{\rightarrow} \Jac^0(\wt C) \to 0.
\end{aligned}
\]
Consider the following commutative diagram with exact rows:
\be \label{irreducible case}
\xymatrix{
0 \ar[r] & \prod_{i=1}^\delta \C^* \ar[r] \ar@{^{(}->}[d]^{\zeta} & \ker(1+\iota^*) \ar[r]^{\beta} \ar@{^{(}->}[d] &\ker(1+\wt \iota^{\,*}) \ar@{^{(}->}[d] & \\
0 \ar[r] & \prod_{i=1}^\delta \C_{p_i}^* \times \C_{\iota p_i}^* \ar[d]^\gamma \ar[r] & \Jac^0(D) \ar@{->>}[d]^{(1+\iota^*)} \ar[r]^{n^*} & \Jac^0(\wt D) \ar[r] \ar@{->>}[d]^{1+\wt \iota^{\,*}} & 0\\
0 \ar[r] & \ker(\alpha) \ar[r] & (1+\iota^*)\Jac^0(D) \ar[r]^{\alpha} & (1+\wt \iota^{\,*}) \Jac^0(\wt D) \ar[r] & 0
}
\ee
Since $\prod_{i=1}^\delta \C^*$ is connected and since $\ker(1+\wt \iota^{\,*}) $ has four connected components (recall that $\wt D$ and $\wt C$ are smooth), to prove that $\ker(1+ \iota^*) $ has four connected components it is sufficient to prove that $\beta$ is surjective, i.e., that $\gamma$ is surjective.
By Proposition 2.14 of \cite{Sacca12}, it follows that every $\iota^*$-invariant sheaf on $D$ is the pull-back via $f^*$ of a sheaf on $C$. It follows that
\[
(1+\iota^*)\Jac^0(D) \subset f^*\Jac^0(C),
\]
so that we can consider the following commutative diagram
\[
\xymatrix{
0 \ar[r] & \ker(\alpha) \ar[r] \ar@{^{(}->}[d] & (1+\iota^*)\Jac^0(D) \ar@{^{(}->}[d] \ar[r]^{\alpha} & (1+\wt \iota^{\,*}) \Jac^0(\wt D) \ar@{=}[d] \ar[r] & 0\\
0 \ar[r] &\prod_{i=1}^\delta \C^* \ar[r] & f^*\Jac^0(C) \ar[r] & \wt f^* \Jac^0(\wt C) \ar[r] & 0.
}
\]
Here, the fact that the kernel of the surjection $f^*\Jac^0(C) \to \wt f^* \Jac^0(\wt C)$ is equal to $\prod_i \C^* $ follows from the fact that the kernel of $\Jac^0(C) \to f^*\Jac^0(C)$ and of $\Jac^0(\wt C) \to \wt f^* \Jac^0(\wt C)$ are both equal to $\Z/(2)$.
Notice, however, that $\im \gamma \cong \prod_{i=1}^\delta \C^*$, so that the series of inclusions
\[
\prod_{i=1}^\delta \C^* \subset \ker(\alpha) \subset \prod_{i=1}^\delta \C^*,
\]
is in fact a series of equalities. Thus   $\gamma$ is surjective and, as a consequence,   $\beta$ too is surjective.

To finish the proof, we just need to show that when we take the closure of $\ker(1+\iota)$ in the compactified Jacobian $\overline \Jac^0(D)$, the number of connected components does not change.
It is well known (cf. for example \cite{Oda-Seshadri79}) that, in order to compactify $\Jac^0(D)$,  one first compactifies the $(\prod_{i=1}^\delta \C_{p_i}^* \times \C_{\iota p_i}^*)$-bundle  over $\Jac^0(D)$ to a $(\prod_{i=1}^\delta \P^1\times \P^1)=(\prod_{s\in \operatorname{Sing}(D)} \P^1_s)$-bundle over $\Jac^0(D)$. In loc. cit. Oda and Seshadri show that this bundle is the normalization of $\overline \Jac^0(D)$. Let us denote by $Z$ this normalization.
In order to obtain $\overline \Jac^0(D)$ we identify the boundary components of $Z$ in the following way: for any $2\delta$-uple,
\[
\Xi=\{\varepsilon_1, \varepsilon_1', \dots,\varepsilon_\delta, \varepsilon_\delta' \}, \quad\text{with }\\\ \varepsilon_i, \varepsilon_i'\in\{1,-1\}, \\\\\ {i=1,\dots, \delta} 
\]
we consider the section
\[
s_\Xi:  \Jac^0(\wt D) \to Z,
\]
defined by taking the $0$-section of the $j$-th component of $ \P^1_{p_1}\times  \P^1_{\iota p_1}\times \cdots  \P^1_{p_\delta}\times  \P^1_{\iota p_\delta}$ if the $j$-th component of $\Xi$ is equal to $1$, and taking the $\infty$-section if it is equal to $-1$. The sections corresponding to $\Xi$ and to $-\Xi$ (which are both identified to $\Jac^0(\wt D)$)  are then glued under  a twist by the degree zero line bundle,
\[
L_\Xi:=\O_{\wt D}\left(\sum_i \varepsilon_i (x_i-y_i) +\sum_i \varepsilon_i '\,\wt \iota\,(x_i-y_i)\right).
\]
Now observe that the map $\zeta$ in diagram (\ref{irreducible case}) sends
$(\lambda_1,\dots,\lambda_\delta)$ to $(\lambda_1,\frac{1}{\lambda_1}\dots,\lambda_\delta, \frac{1}{\lambda_\delta})$. Sending the $\lambda_i$'s to zero, or infinity it follows that
the closure of $\ker(1+\iota^*)$ in the normalization of $\overline \Jac^0(D)$ has four connected components, each of which intersects the image of the section $s_{\Xi}$ if and only if
\[
\varepsilon_i=-\varepsilon_i', \quad \text{for} \\\ i=1, \dots, \delta.
\]
In this  the case then, by (\ref{components ker norm}), $L_\Xi$ belongs to the identity component of $\ker(1+\wt \iota^{\,*})$. It follows that tensoring by $L_\Xi$ preserves each connected component of ${\ker(1+\wt \iota^{\,*}) }\subset \Jac^0(\wt D)$, and hence we may conclude that the fixed locus of $\tau$ on $\overline \Jac^0(D)$ has four connected components.
\end{proof}

We highlight the following corollary.

\begin{cor} \label{rank one degeneration}
Let $C$ be an integral curve with one node and no other singularity, and let $D=f^{-1}(C)$ be the corresponding integral curve with two nodes. Then $\Prym (D/C)$ is a rank one degeneration of an abelian variety.
\end{cor}

\subsubsection{A reducible case}\label{prym reducible}
The following example shows that the situation when $C$ and $D$ are not irreducible is slightly different.

Let us consider the case where  the curves $C=C_1\cup C_2$ and $D=D_1\cup D_2$ are union of two smooth components intersecting transversally in $\delta$ (resp. $2\delta$) points. For simplicity we assume $\delta=1$ so that $C_1\cdot C_2=1$.  Set
\[
D_1\cap D_2=\{ p, \iota p\},
\]
and let $\{p_1, \iota p_1\}$ and $\{p_2, \iota p_2\}$ be the pair $\{ p, \iota p\}$  viewed on $D_1$ and  $D_2$ respectively. We consider the case where the polarization is equal to $D$, so that $\overline \Jac^0_D(D)$ is irreducible. Also, for $i=1,2$ let 
\[
\iota_i: D_i \to D_i,
\]
be the involution corresponding to the double cover $D_i \to C_i$, and let $\eta_i$ be the line bundle defining the cover itself.
In this case diagram (\ref{irreducible case}) becomes
\[
\xymatrix{
0 \ar[r] & \C^* \ar[r] \ar@{=}[d] & \ker(1+ \iota^*) \ar[r]^{\beta} \ar@{^{(}->}[d] & \ker(1+\wt\iota^{\,*}) \ar@{^{(}->}[d] &  \\
0 \ar[r] & \C^* \ar[d]^{\gamma=0} \ar[r] & \Jac^0(D) \ar@{->>}[d]^{(1+\iota^*)} \ar[r]^{n^*} & \Jac^0(\wt D) \ar[r] \ar@{->>}[d]^{1+\wt\iota^{\,*}} & 0\\
0 \ar[r] & \ker(\alpha) \ar[r] & (1+\iota^*)\Jac^0(D) \ar[r]^{\alpha} & (1+\wt\iota^{\,*}) \Jac^0(\wt D) \ar[r] & 0
}
\]
and gives an exact sequence
\[
0 \to \C^* \to \ker(1+ \iota^*) \stackrel{\beta}{\rightarrow} \ker(1+\wt\iota^{\,*}) \to \ker(\alpha) \to 0,
\]
so that to determine the number of connected components of $\ker(1+\iota^*)\subset \Jac^0(D)$ we need to compute $\ker(\alpha)$.  Since  in the case at hand, $\Jac^0(\wt D)\cong \Jac^0(D_1)\times \Jac^0(D_2)$, the fixed locus $\ker(1+\wt\iota^{\,*})$ has $16$ connected components.
We claim that, $\ker(1+\iota^*)$ has $8$ connected components. By what we said, this is equivalent to showing that $\ker(\alpha)\cong \Z/(2)$. To this aim, consider the following diagram
\[
\xymatrix{
0 \ar[r] & \ker(\alpha) \ar[r] \ar@{^{(}->}[d] & (1+\iota^*)\Jac^0(D) \ar@{^{(}->}[d] \ar[r]^{\alpha} & (1+\wt\iota^{\,*}) \Jac^0(\wt D) \ar@{=}[d] \ar[r] & 0\\
0 \ar[r] & \ker(\rho)\cong \Z/(2) \ar[r] & f^*\Jac^0(C) \ar[r]^\rho & \wt f^* \Jac^0(\wt C) \ar[r] & 0\\
& 0 \ar[r] &\Jac^0(C) \ar[r]^\sim \ar@{->>}[u] & \Jac^0(\wt C) \ar[r] \ar@{->>}[u] &0 \\
& 0 \ar[r] &\Z/(2) \ar[u] \ar[r] & \Z/(2)\times \Z/(2) \ar[r]\ar[u] & \Z/(2).
}
\]
Suppose for a moment that $ \ker(\alpha)=(0)$, then $\alpha$ would be  an isomorphism and $\alpha^{-1}$ would give a section of the non trivial covering  $f^*\Jac^0(C) \to \wt f^* \Jac^0(\wt C)$. This, however, is absurd and hence
\[
\ker(\alpha)=\ker(\rho)\cong \Z/(2),
\]
and the claim is proved.

Finally we  observe that the closure in $\overline \Jac^0_D(D)$ of $\ker(1+\iota^*)$ is the union of $4$ connected components, each of which is the union of two irreducible components.
What we will show is that the closure every connected component of $\ker(1+ \iota^*)$ intersects the closure of exactly one other connected component.
Indeed, $\overline \Jac^0_D(D)$ is obtained from the $\C^*$-bundle over $\Jac^0(\wt D)=\Jac^0(D_1) \times \Jac^0(D_2)$ by first compactifying to a $\P^1$-bundle and then glueing the $0$ and $\infty$ sections, (both identified with $\Jac^0(\wt D)$),  via the twist by  line bundle  $L:=\O_{D_1}(p_1-\iota p_1) \otimes \O_{D_2}(p_2-\iota p_2)$. However,  by  (\ref{components ker norm}), the class of $L_1:=\O_{D_1}(p_1-\iota p_1) $ does not lie in the the identity component of $\ker(1+\iota^*_1)$, even though it lies $\operatorname{Nm}^{-1}(0) \subset \Jac^0(D_1)$. It follows that tensoring by $L_1$ preserves the two fibers $\operatorname{Nm}^{-1}(0)$ and $\operatorname{Nm}^{-1}(\eta_1)$ (cf.  (\ref{ker norm})), while interchanging the two components of each fiber. The analogous statement holds for $L_2:=\O_{D_2}(p_2-\iota p_2)$.
We may conclude that tensoring by $L^{\otimes 2}$ preserves every component of $\ker(1+\wt\iota^{\,*})=\ker(1+\iota^*_1) \times \ker(1+\iota^*_2)$. The action of tensoring by $L$ thus divides these components in pairs, and within each pair the zero and infinity sections of the two components are identified.

A final remark can be made about the points at infinity of $\Fix(\tau)$. When the parameter $\lambda$ of the $\C^*$-bundle $\ker(1+ \iota^*)$ goes to zero or  infinity the corresponding line bundle
tends to a torsion free sheaf which is $\mc S$-equivalent to a polystable sheaf of the form $F_1\oplus F_2$ where, for each $i=1,2$, the sheaf $F_i$ is a stable sheaf supported on $D_i$

The previous discussion can be repeated almost word by word in the case in which 
the two curves $C_1$ and $C_2$ meet transversally in $\delta\geq 1$ points. The main difference in the case $\delta \ge 2$ is that $\overline \Jac^0_{D}(D)$ is not irreducible any more. However, on can see (cf. \cite{Sacca13}) that among the $2\delta-1$ components of $\Jac^0_{H_D}(D)$ only one them, the identity component, contains $\ker(1+\iota^*)$.

\begin{prop} \label{polyst}Let
 $C=C_1\cup C_2$ (resp. $f^{-1}(C{)}=D=D_1\cup D_2$) be the  union of two smooth components intersecting transversally in $\delta$ (resp. $2\delta$) points.
Then $\Fix(\tau)\subset \overline \Jac^0_D(D)$ has $4$ connected maximal dimensional components, each consisting in two irreducible components meeting at the boundary of 
$ \overline \Jac^0_D(D)$. Moreover each connected component of 
$\Fix(\tau)$ contains points of the type $[F_1\oplus F_2]$ where 
for each $i=1,2$ the sheaf $F_i$ is a stable sheaf supported on $D_i$.
\end{prop}

\vskip 0.5 cm 

\subsection{Changing the involution and the  polarization} \label{choose_pol}

As we observed after the definition \ref{def_rel_prym} of relative Prym variety,
the choice of the polarization $H$ appears to lead to a dichotomy: either the ambient space 
$M_{v,H}$ is smooth and the involution $\tau$ is not regular, or $\tau$ is regular
and the ambient space is smooth. In this section we will introduce a ``twisted'' version of the involution $\tau$
and prove that, also in this more general setting, the above dichotomy can not be reconciled. 

Start with a Mukai vector
\[
v=(0,[D], \chi)\in H^*(S,\mathbb Z)
\]
where $\chi$ is not necessarily equal to $-g(D)+1$. 

The first remark is that Lemma \ref{dual} implies that \emph{for any} line bundle $N$ on $S$, we have
$$
\Shext^1_S(F, N)\cong  \Shhom_\Gamma(F, N\otimes {\mc O}_\Gamma(\Gamma))\,.
$$
We set
$$
j_N(F)=\Shext^1_S(F, N)
$$
We now proceed, step by step, exactly as in 
Subsection \ref{rel_prym_sub}. First  of all we notice that  $j_N^2 (F)\cong F$, then we set
\be\label{tau_n}
\tau_N=j_N\circ\iota^*\,,\qquad \text{i.e.}\quad\tau_N(F)=\Shext^1(\iota^*F,N)
\ee
With this notation we have $\tau=\tau_{-D}$.
Again, by virtue of Lemma \ref{base change ext}, we notice that, if conditions a) b) and c) of (\ref{cond_abc}) are satisfied with $\tau$ replaced by $\tau_N$, then
the assignment: $F\longmapsto \tau_N(F)$
yields a well defined involution 
\be\label{tau_N_moduli}
\tau_N: M_{v, H}\longrightarrow M_{v, H}
\ee
Condition a) is satisfied if $N$ is $\iota^*$-invariant.
For condition b) 
 we demand
\be\label{mukai_j_N}
c_1(F)=c_1(j_N(F))=c_1(\Shext^1_S(F, N))\,,\qquad \chi(F)=\chi(j_N(F))
\ee
The first of these two conditions is always satisfied,
 as for the second, exactly as in (\ref{hilb_j_N}), we have
\be\label{hilb_j_N}
p_{j_N(F)}(m, H)= -p_{F\otimes N^\vee }(-m, H).
\ee
and in particular for any pure sheaf of dimension one,
\be \label{chi j F}
\chi(j_{N}(F))=-\chi(F\otimes N^\vee  )=-\chi(F)+N\cdot c_1(F).
\ee
When $c_1(F)=D$, then  $\chi(j_N(F))=\chi(F)$ if and only if
\be \label{chi=chi'}
2\, \chi(F)=N\cdot D.
\ee

\vskip 0.3 cm

Let us take a closer look at the involution $\tau_N$. As for the case of $\tau=\tau_{-D}$ even if condition c) of  (\ref{cond_abc})  is not satisfied we still have a rational map
$$
\tau_N: M_{v, H}  \dasharrow M_{v, H}\,.
$$
Therefore, the natural question is whether there exists a pair $(N, H)$ satisfying conditions a), b), and c) of  (\ref{cond_abc}) (making $\tau_N$ regular), and such that, moreover, $H$ is $v$-generic, (making $M_{v,H}$ smooth).
In this subsection we show that such a pair {\it does not exist}. It follows that if we choose $H$ to be $v$-generic, it is not a priori clear whether the birational map $\tau_N$ extends to a regular morphism. In Section \ref{singular} we will show that $\tau_N$ does not extend when $|C|$ is non-hyperelliptic. In Section
 \ref{hyp} we will show that it does when $|C|$ is hyperelliptic.

 \begin{prop}\label{never_preserve}Let $T$ be an Enriques surface and $f: S\to T$ its universal cover and set $D=f^{-1}C$ for some irreducible curve  $C\subset T$. Consider a non zero integer $\chi$ such that $v=(0, [D], \chi)$ is primitive. Suppose that there exists a member $D_1+D_2$ of $f^*|C| \subset |D|$ that is the union of two integral curves $D_1$ and $D_2$
intersecting transversally and whose classes are $\iota^*$-invariant. Let $N$ be a line bundle on $S$, satisfying  a) and b) of (\ref{cond_abc}). 

If $H$ is a $v$-generic polarization on $S$, then $N$ then $(N, H)$ does not satisfy c) of (\ref{cond_abc}).
\end{prop}
Before proving  Proposition \ref{never_preserve} we need to establish the following lemma.
 
\begin{lem} \label{restriction of j(F)} 
Consider a curve
\[
D_1 +D_2,
\]
in $S$ that is the union of two integral curves $D_1$ and $D_2$ meeting transversally, and let $F$ be a pure dimension one sheaf on such a curve. Let $N$ be a line bundle on $S$. Referring to {\rm Notation \ref{Fi}},
we have
\[
(j_N(F))_j=j_N(F_j\otimes {\mc O}(-\Delta_F)), \,\,\, j=1, 2,
\]
where, as in Lemma \ref{Fj}, $\Delta_F \subset D_1 \cap D_2$ is the set of nodes where $F$ is locally free.
\end{lem}

\begin{proof} Let $i\neq j$, and consider the short exact sequence
\be \label{short exact sequence}
0 \to F^j \to F \to F_i \to 0.
\ee
From Lemma \ref{Fj} it follows that
\[
F^j\cong F_j\otimes {\mc O}(-\Delta_F).
\]

Applying $j_N(\cdot)$, we get
\[
0 \to j_N(F_i) \to j_N(F) \stackrel{a}{\to} j_N( F_j\otimes {\mc O}(-\Delta_F)) \to 0.
\]
Moreover, since $ j_N( F_j\otimes {\mc O}(-\Delta_F)) $ is torsion free and supported on $D_j$, the morphism $a$ factors through a surjective morphism
\[
(j_N(F))_j \to j_N( F_j\otimes {\mc O}(-\Delta_F)).
\]
However, this morphism is also injective because it is a generic isomorphism, and thus the Lemma is proved.
\end{proof}

Clearly, if $F$ is locally free this Lemma is just the consequence of the fact that restriction to a subcurve is a group homomorphism of the Picard groups.

\begin{proof}{\sl (of Proposition \ref{never_preserve})}
Let  $g_1$ and $g_2$ be the genera of $D_1$ and $D_2$ respectively. Let $F$ be a pure sheaf of rank one on $D$, and as usual set $F_i=F_{D_i}$, for $i=1,2$. Set
\[
\chi=\chi(F),  \quad \chi_i=\chi(F_i),\,\,\,\text{for} \,\,i=1,2 \quad \text{and} \,\,\, q=\chi(Q),
\]
where $Q$ is the cokernel of the natural injection $F \to F_1\oplus F_2$. Then
\[
\Delta_F:=\supp(Q),
\]
is the locus of $D_1\cap D_2$ where $F$ is locally free.
Furthermore, let $H$ be a polarization, and set
\[
k_1=H\cdot D_1 , \,\,\, k_2=H\cdot D_2 \quad \text{and} \,\,\, k=H\cdot D=k_1+k_2.
\]
By Lemma \ref{stability sub curves}, we know that $F$ is $H$-stable if and only if 
\[ 
\Bigg\{\aligned  
&\ff{\chi}{k}<\ff{\chi_1}{k_1}, \\
&\ff{\chi}{k}<\ff{\chi_2}{k_2}.
\endaligned
\]
Equivalently, since $\chi=\chi_1+\chi_2-q$, the sheaf $F$ is $H$-stable if and only if 
\be \label{stability of F}
\ff{\chi}{k}<\ff{\chi_1}{k_1}< \ff{\chi}{k}+\ff{q}{k_1}.
\ee
Now, using Lemma \ref{restriction of j(F)} and formula (\ref{chi j F}), setting $n=D\cdot N$ and $n_i=D_i\cdot N$, for $i=1, 2$, we see that $j_N(F)$ is $H$-stable if and only if,
\[
-\ff{\chi}{k}+\ff{n}{k}<-\ff{\chi_1}{k_1}+\ff{q}{k_1}+\ff{n_1}{k_1}< -\ff{\chi}{k}+\ff{n}{k}+\ff{q}{k_1}.
\]
Suppose now that $\tau_N$ satisfies condition b) in  (\ref{cond_abc}), i.e. that $N$ is such that $2\chi(F)=N\cdot D$, then this last string of inequalities becomes,
\be \label{stability of jF}
\ff{\chi}{k}<-\ff{\chi_1}{k_1}+\ff{q}{k_1}+\ff{n_1}{k_1}<\ff{\chi}{k}+\ff{q}{k_1}.
\ee

On the other hand, that if $H$ is $v$-generic, then
\[
m=\ff{\chi}{k}{k_1},
\]
is not an integer. Indeed, if $m$ is an integer, we can set $\chi_1=m$, $\chi_2=\chi-m+q$, and find two $H$-stable sheaves $F_1$ and $F_2$ supported on $D_1$ and $D_2$ respectively, with $\chi(F_1)=\chi_1$ and $\chi(F_2)=\chi_2$. But then the sheaf $F_1\oplus F_2$ is $H$-polystable sheaf with Mukai vector $v$.

Let $a$ be the round down of $m$, so that we can write $m=a+s$ with $0<s<1$. Then (\ref{stability of F}) and (\ref{stability of jF}) become respectively,
\[
a+1\le \chi_1 \le a+q, \quad \text{and} \quad  a+1\le -\chi_1+n_1+q \le a+q,
\]
meaning that every $\chi_1$ satisfying the first two inequalities must also satisfy the remaining two. We then get 
\[
n_1=2a+1.
\]
The proposition follows noticing that since $N$ is $\iota^*$-invariant,  $n_1=N\cdot D_1$ has to be even.
\end{proof}

A few remarks are in order. First of all,  relaxing the requirement that $H$ be $v$-generic, we can indeed find a pair $(H,N)$ satisfiying conditions a), b) and c) in  (\ref{cond_abc}). For example, if $\pm N$ is ample, then we can choose $H=N$. Second, we highlight a corollary of the proof of Proposition \ref{never_preserve}.
 
\begin{cor} \label{existence polystable}
Consider the set up and the notation of Proposition \ref{never_preserve}. Let $H$ be a polarization on $S$ and suppose that there exist a line bundle $N$ on $S$ satisfying condition b) and  c) in (\ref{cond_abc}). Then there is an $H$-polystable sheaf with Mukai vector $v$ of the form $F_1\oplus F_2$ where for $i=1,2$, $\supp(F_i)=D_i$ and $F_i$ is $H$-stable.
\end{cor}
\begin{proof}
Indeed, using the notation as in the proof of Proposition \ref{never_preserve}, one can see that if $D\cdot N=2\chi$, and if $\chi k_1/k$ is an integer then, forcing (\ref{stability of F}) and (\ref{stability of jF}), yields,
\be \label{N e H}
\frac{D\cdot N}{D\cdot H}=\ff{n}{k}=\ff{n_i}{k_i}=\frac{D_i\cdot N}{D_i\cdot H}, \quad \text{for} \,\,\, i=1,2.
\ee
It follows that there exists an $H$-polystable sheaf $F=F_1\oplus F_2$, as requested, by choosing $\chi(F_i)=\ff{\chi k_i}{ k}=\ff{n_i}{2}$.
\end{proof}

Notice that for  the corollary to be true, we do not really need to require that $j_N$ be a regular morphism, but  only that it is regular in a neighborhood of $\pi^{-1}[D_1+D_2]$.

Of course, one could give a definition of the relative Prym variety based on the involution $\tau_N$, but a word of caution 
is in order. When $N\neq-D$, in general one can not assume that the image of the rational zero-section of $\Jac^0(|D|)\to |D|$
lies in the fixed locus of $\tau_N$. As a consequence there is no privileged irreducible component of 
$\ov{\Fix(\tau_N)}$. One could then content oneself in saying that any irreducible, maximal dimensional component of 
$\ov{\Fix(\tau_N)}$ is a relative Prym variety. Also in this context one could prove that an open subset of the regular part of the relative Prym variety carries a symplectic structure, and also in this context one could prove that, when $|D|$ is non-hyperelliptic,
the relative Prym variety admits no symplectic resolution. Thus nothing much is to be gained in putting oneself in this more general context. For this reason and also to lighten the exposition, in what follows we will only consider the case
$N=-D$.


\section{Kuranishi families and tangent cones 
}\label{kuranishi}

Let $T$ be an Enriques surface and $f: S\to T$ its universal covering. Let $D=f^{-1}(C{)}$ for some integral curve  $C\subset T$ of genus $g\geq2$. 
Consider the Mukai vector $v=(0, [D], -h+1)$, with $h=g(D)$.  Choose a polarization $H$ 
such that:

- {\it  There exists a point $[F]\in M_{v,H}$, with $F=F_1\oplus F_2$, where $F_1$ and $F_2$ are two non isomorphic $H$-stable (in particular $H$
is not $v$-generic).  }

Polystable sheaves of this type exist (cf.  Corollary \ref{existence polystable}) and correspond to the simplest singularities that can appear in $ M_{v,H}$.
Choose a divisor $N$ satisfying conditions 
a)  and b)  in (\ref{cond_abc}).
In this section we will study the tangent cones to $P_{v,H}$ and to $M_{v, H}$  at  singularities of the above type. To simplify  we adopt the following notation:
\be
\label{notation_kur}
M=M_{v, H}\,,\qquad \tau=\tau_N: M\dasharrow  M\,,\qquad P=P_{v,H}=\overline{\operatorname{Fix}^0(\tau)}
\ee
Given a point $[F]$ belonging to $M$ (to  $P$),  we denote by 
$$
\mc M\subset M\,\quad (\text {resp.}\quad \mc P\subset P)
$$
a suitable analytic neighborhood of $[F]$ in $M$ (resp. in $P$). Furthemore, when $[F]\in P$, we will always assume that $\tau$ is defined in a neighborhood of $[F]$.   The main example of this situation is the case:
$H=D$.

We next recall a few fundamental facts
about Kuranishi families. 

\vskip 0.5cm

\subsection{Kuranishi families}

To begin with,  let us look at any point $[F]\in M$, where $F$ is a polystable sheaf. Let us consider a Kuranishi family for  $F$
parametrized by a pointed analytic scheme $(B,b_0)$:

\be\label{kuran_di_base}
\xymatrix{\F\ar[d]^\xi\\
 S\times B}\qquad\qquad F=\F_{b_0}=\xi^{-1}(b_0)
\ee
Set
\[
G=\PP\Aut(F)
\]
By the universal property of the Kuranishi family, the group $G$ acts on $(B,b_0)$ and 
an analytic neighborhood $\mc M$ of $[F]$ in $ M$ may be identified with the quotient of $B$ by $G$:
\[
B\sslash G=\mc M\subset M_{v,H}
\]
Moreover, by Luna's slice \'etale theorem, the analytic space $B$ is algebraic in the following sense. Let $Q^{ss} \subset \operatorname{Quot}$ be the open subset parametrizing points $x=[{\mc O}_X(-kH)^{\oplus m}\to F]$ such that $F$ is $H$-semistable and such that the induced map $H^0({\mc O}_X^{\oplus m}) \to F (kH)$ on global sections is an isomorphism, and write 
\[
M=Q^{ss}/\PP GL(m),
\]
There is a point $x=[{\mc O}_X(-kH)^{\oplus m}\to F]\in Q^{ss}$
and a $G$-invariant subscheme $\mathcal T\subset Q^{ss}$ passing through $x$,
such that $\mathcal T\sslash G\to M$ is \'etale and $B$ is an analytic neighborhood of $x$ in $\mathcal T$. 
To study the tangent cones to $B$ at $b_0$  and  of $M$ at $[F]$,  it is best to study the completions $\wh{\mc O}_{B,b_0}$ and $\wh{\mc O}_{M,[F]}$, as 
these two rings can be efficiently
 described in terms of the {\it Kuranishi map}.  We follow the notation of 
 \cite {Kaledin-Lehn-Sorger06} and of \cite{Lehn-Sorger06}.
 Let $\Ext^2(F,F)_0$ be the kernel of the trace map  $\Ext^2(F,F) \to H^2(S, \O_S)$. The Kuranishi map
 is a formal map
\[
\kappa: \Ext^1(F,F)\longrightarrow\Ext^2(F,F)_0
\]
starting with a quadratic term: $\kappa=\kappa_2+\kappa_3+\cdots$.
(For a beautiful and very explicit construction of the Kuranishi family see the appendix in  \cite{Lehn-Sorger06})

The Kuranishi map has the following properties: 

a) $\kappa$ is equivariant under the natural action of $G$ on $\Ext^1(F,F)$ and on $\Ext^2(F,F)_0$.

b) $\kappa^{-1}(0)$ is (isomorphic to) a formal neighborhood of $b_0$ in $B$, while  and 
$\kappa^{-1}(0)\sslash G$ is (isomorphic to) a formal neighborhood of $[F]\in M$, 

c) The quadratic part $\kappa_2$ of $\kappa$, which is called the {\it moment map},  is given by the cup product
\be\label{moment}
\aligned
\kappa_2:\,\,\,  \Ext^1(F,F)&\longrightarrow\Ext^2(F,F)_0\\
&e\mapsto\kappa_2(e)= \frac{1}{2}e\cup e
\endaligned
\ee
From now on we will set 
\[
Q=\kappa_2^{-1}(0)
\]

Suppose now that  the point $[F]$ belongs to  $P$. Since we are assuming that $F$ is $H$-polystable, we have $F\cong\tau(F)$.
The symplectic involution $\tau$ on $M$ lifts (non uniquely) to an automorphism on the parameter space $B$ of the Kuranishi family (\ref{kuran_di_base}).
 Indeed, given such a family and fixing
an isomorphism
\[
\phi: F\cong\tau(F)\,.
\]
 we  get a new family 
 by applying $\tau$ to it
\[
\xymatrix{\tau(\F)\ar[d]^{\tau(\xi)}\\
 S\times B
}\,\qquad\qquad \varphi^{-1}: F\overset\cong\to\tau(F)=\tau(\F)_{b_0}
\]
and therefore, by universality, we obtain an automorphism 
\[
 \tau_\phi: B\to B,
\]
 whose first order term is uniquely defined.
This automorphism need not be an involution, but it is so at the infinitesimal level since 
\be\label{dtau}
d\tau_\phi=d\tau=\tau_*: \Ext^1(F,F)\longrightarrow \Ext^1(F,F)\,.
\ee
Sometimes, and when no confusion is possible, we will write $\tau$ instead of $\tau_*$ to indicate the 
homomorphism (\ref{dtau}).

\begin{rem} \label{tau_ext2}{\rm If $F$ is a stable sheaf, the action of $\tau_*$ on $\Ext^2(F, F)$ is equal to $-1$.
To see this it suffices to prove that 
$\tau_*(e\cup f)=-e\cup f$ for $e, f \in \Ext^1(F,F)$. 
Interpreting the cup product in terms of composition of short exact sequences,
we see that $\tau_* (e \cup f)= \tau_*(f) \cup \tau_*(e)$. The fact that $\tau$ is symplectic tells us that:
$\tr(e \cup f)=tr(\tau_*(e)\cup \tau_*(f))=\tr(\tau_*(f\cup e))$. 
Since  $\tau_*(f\cup e)=\lambda f\cup e$, $\lambda \in \C^*$, and 
$\tr(e \cup f)=-\tr(f\cup e)$, we get 
 $\lambda=-1$.}
\end{rem}

\vskip 0.5cm

\subsection{Tangent cones}

We now turn our attention to tangent cones. Let $C_{b_0}(B)$ and $C_{[F]}(M)$,   denote the tangent cones to $B$ at $b_0$ and   to $M$  at  $[F]$, respectively. Again we will assume that $F$ is polystable.
For simplicity write  
\[
C(B)=C_{b_0}(B)\,,\qquad C(M)=C_{[F]}(M)
\]
The description of  $C(B)$ and  $C(M)$ is particularly simple 
under the condition that the stable decomposition of $F$
\be\label{poly_F}
F=F_1\oplus F_2
\ee
has only two non isomorphic summands. Indeed this implies that
\be
\label{condition_star} \dim \Ext^2(F,F)_0=1,
\ee
that $G=\C^*$ and that
$B$, or better its formal neighborhood at $0$, is given by a single equation
$\kappa=0$ and, by point c) above, we have
\be\label{cono_B}
C(B)=Q=\kappa_2^{-1}(0)=\{e\in\Ext^1(F,F)\,\,|\,\,e\cup e=0\}
\ee
On the other hand, by point a),  $\kappa$ and $\kappa_2$ must be $\C^*$-invariant, 
 hence
\be\label{cono_M}
C(M)=Q\sslash \C^* \,\,\,\subset \,\,\Ext^1(F,F)\sslash \C^*\,,
\ee

We use the  non-degenerate bilinear form 
\be\label{cup_bilin}
\aligned
\mu: \Ext^1(F_2, F_1)\times&\Ext^1(F_1, F_2)\longrightarrow \C\\
&(f,f')\longmapsto\mu(f,f')=\tr(f\cup f')
\endaligned
\ee
to identify  $\Ext^1(F_2, F_1)$ with $\Ext^1(F_1, F_2)^\vee$ and we write
\be\label{u_w}
\aligned
U_1&:=\Ext^1(F_1, F_1)\,,\quad U_2:=\Ext^1(F_2, F_2)\,,\\
W&:=\Ext^1(F_1, F_2)\,,\quad W^\vee=\Ext^1(F_2, F_1)\,,
\endaligned
\ee
  We identify   $\Ext^2(F,F)_0$ with $\Ext^2(F_1,F_1)$, and  $\Ext^2(F_1,F_1)$ with $\C$ (via the trace).
Up to the constant factor $\frac{1}{2}$, the  moment map (\ref{moment}) is given by
\be\label{moment2}
\aligned
\Ext^1(F,F)=U_1\oplus U_2\oplus W^\vee 
 \oplus W&\longrightarrow \Ext^2(F,F)_0=\C\\
 e=(a, b, f, f')&\longmapsto \kappa_2(e) =\mu(f, f')\,.
 \endaligned
\ee
\begin{rem}
Using the notation of  \cite{Nakajima98}, p. 520, we may also write the non-degenerate bilinear form (\ref{cup_bilin}) as a Nakajima's moment map
\be\label{cup_bilin_nakaj}
\aligned
\mu: \Hom(W, \C)\oplus&\Hom(\C,W)\longrightarrow \C\\
&(i,j)\longmapsto\mu(i,j)=i j
\endaligned
\ee
 based on the trivial quiver whose graph consists of a single vertex $x$ and  
no edges, and the pair of vector spaces attached to $x$ is the pair $(V, W)$ with $V=\C$. We will freely pass from notation (\ref{cup_bilin}), to (\ref{cup_bilin_nakaj}) and viceversa.
\end{rem}
The action of $\C^*$ on $\Ext^1(F,F)$ is given by
\[
\lambda\cdot(a, b, f, f')= (a, b, f\lambda^{-1}, \lambda f')
\]
or, in Nakajima's notation, $\lambda\cdot (i,j)=(i\lambda^{-1}, \lambda j)$.
The natural map 
\[
\aligned
\Hom(W, \C)\oplus&\Hom(\C,W)\longrightarrow \End(W)\\
&(i,j)\longmapsto ji
\endaligned
\]
factors through the action of $\C^*$ and exhibits the quotient of  $\Hom(W, \C)\oplus\Hom(\C,W)$  by $\C^*$ as the set $ \End_1(W)$ of endomorphisms of $W$ of rank $\leq 1$. Thus,
\be\label{quot_ext}
\aligned
\Ext^1(F,F)/\C^*&\cong\left(U_1\oplus U_2\oplus\Hom(W, \C)\oplus\Hom(\C,W)\right)/\C^*\\
&\cong U_1\times U_2\times \End_1(W)
\endaligned
\ee
By choosing appropriate coordinates on $\Ext^1(F,F)$ and looking at $\C^*$-invariant polynomial functions on this vector space, one may deduce that the equation $k_2=0$ of $Q/\C^*\subset \Ext^1(F,F)/\C^*$ is   {\it linear}  in the invariant coordinates and is  of the form,
\be\label{kuran_eq}
A=0\,,\quad \text{where } \quad A\in W\otimes W^\vee=\End(W)^\vee\,.
\ee
In particular, $Q/\C^*$ is irreducible and of the same dimension as $C(M)$ so that,
\[
C(M)=Q/\C^*=U_1\times U_2\times \End_{1,A}(W)
\]
where
\[ \label{end1h}
 \End_{1,A}(W)=\{a\in
 \End_{1}(W)\,|\, A(a)=0\}.
\]

Let us now examine  the case in which  the point $[F]$, corresponding to the sheaf (\ref{poly_F}), lies in $P$, so that $F\cong\tau(F)$.
Let $D_1$  and $D_2$, be the supports  of $F_1$ and $F_2$, respectively. From now on  we proceed under the assumption that  $D_1$  and $D_2$ are two $\iota$-invariant integral curves meeting transversally. In particular,
\be \label{iota_inv}
\tau(F_i)\cong F_i\,,\qquad i=1,2\,.
\ee

Under these hypotheses, 
 we wish to describe the tangent cone
 $C{(}P)=C_{[F]}{(}P)$ to $P$ at $[F]$. We have
\be\label{cone_P}
 C{(}P)=C(M^\tau)\subseteq C(M)^\tau\subseteq (\Ext^1(F,F)/ \C^*)^\tau,
 \ee
 where the action of $\tau$ on $(\Ext^1(F,F)/ \C^*$ is the one induced by (\ref{dtau}) and where  $X^\tau$ stands for $\operatorname{Fix}_\tau(X)$, the fixed locus of $\tau$ in $X$. We will momentarily see that (\ref{cone_P}) is in fact a series of equalities.
Looking at (\ref{quot_ext}) we have

\be\label{quot_ext_tau}
(\Ext^1(F,F)/ \C^*)^\tau \cong U_1^\tau\times U_2^\tau\times\End_1(W)^\tau
\ee
Now
\[
\End_1(W)^\tau=\{\phi\otimes\tau(\phi)\ |\ {\phi\in W^\vee} \}/\C^*\subset \End_1(W) ,
\]
and this can be identified with the set $\End_1(W)^s $     of symmetric endomorphisms with respect to the non degenerate bilinear form on $W$ defined by
\[
B(v,w)=\tau^{-1}(v)(w)
\]
As is well known, $\End_1(W)^s $ is isomorphic to the affine cone over the degree two
Veronese embedding
\[
\PP W\hookrightarrow\PP S^2W.
\]
To prove that (\ref{cone_P}) is a sequence of equalities it suffices to show that
the right most term in (\ref{cone_P}) is irreducible of dimension equal to $g-1=\dim P$.
Since $\End_1(W)^s $ is irreducible we must only care about the dimensionality statement. 
Let $h_1=2g_1-1$ and $h_2=2g_2-1$ be the genera of $D_1$ and $D_2$ respectively.
The stability of $F_i$, $i=1,2$, and their $\tau$-invariance (\ref{iota_inv})
tell us that $[F_1]$ and $[F_2]$ are smooth points of relative Prym varieties of dimensions $g_1-1$ and $g_2-1$ respectively (cf. Remark (\ref{intr_cones}) below) so that $\dim U_i=2g_i-2$ for $i=1,2$. It follows that
\[
\dim(\Ext^1(F,F)/ \C^*)^\tau=2g_1-2+2g_2-2+\dim\End_1(W)^s=h_1+h_2-2+\dim W
\]
We must then compute the dimension of $W=\Ext^1(F_1,F_2)$. From the isomorphism
\[
\E xt^1(F_1, F_2)\cong\underset{p\in D_1\cap D_2}\oplus\underline\C_p
\]
and from the local to global spectral sequence we get
\be\label{dim_ext_mist}
\Ext^1(F_1, F_2)=H^0(S, \E xt^1(F_1, F_2))=\C^{D_1\cdot D_2}\,.
\ee
Hence
\[
\dim(\Ext^1(F,F)/ \C^*)^\tau=h_1+h_2-2+D_1\cdot D_2 =\frac{1}{2}D^2=h-1=\dim P
\]

\begin{rem} {\rm In proving that $C({P})=(\Ext^1(F,F)/ \C^*)^\tau$ we implicitly proved that the 
quadratic part of the Kuranishi  equation vanishes identically on $C({P})$, i.e. that the form $A$
given in (\ref{kuran_eq}) vanishes identically on $\End_1(W)^s$. This follows directly from the fact 
that $\kappa_2$ is $\tau$-equivariant and that $\tau$ acts as $-1$ on $\Ext^2(F, F)$ (cf. Remark \ref{tau_ext2}), so that}
\[
A(\phi\otimes\tau(\phi))=\phi\cup\tau(\phi)=-\tau(\phi\cup\tau(\phi))=-\phi\cup\tau(\phi)=0
\]
\end{rem}

We summarize the results obtained in this section in the following proposition.

\begin{prop}\label{cones} The assumptions and the notation being the ones introduced at the beginning of this section,  let
 $H$ be a polarization such that there exists an $H$-polystable sheaf $[F]\in M=M_{v,H}$ of the form
$F=F_1\oplus F_2$, with the $F_1$ and $F_2$ two non isomorphic $H$-stable sheaves. For $i=1,2$, let $D_i$ be the support of $F_i$.
Assume that $D_1$ and $D_2$ are integral $\iota$-invariant curves meeting transversely and let $h_i$ be the genus of $D_i$,  $i=1,2$. Set $W=\Ext^1(F_1, F_2)$. Then

a) The tangent cone
$C_{[F]}(M)$ to $M$ at $[F]$ is isomorphic to $\C^{2(h_1+h_2)}\times\End_{1,A}(W)$, where $\End_{1,A}(W)$
is defined by (\ref{end1h});

b) Suppose that $H$ is such that $\tau: M_{v,H} \dashrightarrow M_{v,H}$ is biregular in a neighborhood of $[F]$ (e.g. if $H=D$, )).
If $[F]\in P=P_{v, H}$, then the tangent cone
$C_{[F]}({P})$ to $P$ at $[F]$ is isomorphic to $\C^{h_1+h_2-2}\times\End_1^s(W)$, where $\End_1^s(W)$
denotes the set of symmetric (w.r.t. a suitable bilinear form) endomorphisms of $W$ of rank $\leq 1$.
\end{prop}

\begin{rem}{\rm We could drop the integrality assumption of $D_1$ and $D_2$ and just require that they have no common components. In fact one can check that also  in this case (\ref{dim_ext_mist}) holds, which is all that matters in the proof of the proposition above.}
\end{rem}
\begin{rem}\label{intr_cones}
{\rm Keeping the notation of of Proposition \ref{cones},
Set  $\chi_i=\chi(F_i)$, $v_i=(0,[D_i],\chi_i)$, $M_i=M_{v_i, H}(S)$,
 $C_i=D_i/\iota$, and $P_i=P_{v_i, H}(|D_i|/|C_i|)$, $i=1,2$.  By hypothesis  $[F_i]$ is a smooth point for both $M_i$ and $P_i$, $i=1,2$ so that $[F]$ is a smooth point in both $M_1\times M_2$ and $P_1\times P_2$. Thus we can write the above description of tangent cones in a more intrinsic way
 \be\label{more_intr}
 \aligned
 C_{[F]}(M)&\cong T_{[F]}(M_1\times M_2)\times\End_{1,A}(W)\,,\\
 C_{[F]}({P})&\cong T_{[F]}(P_1\times P_2)\times\End_1^s(W)\,
 \endaligned
 \ee
}
\end{rem}


\section{Analysis of  singularities and the non hyperelliptic case}\label{singular}

For the next proposition we keep the notation and the hypotheses introduced at the beginning of the preceding Section \ref{kuranishi}.
Here we look more closely at the singular points of $M$
and $P$, whose tangent cones were described in Proposition \ref{cones}. The first result we prove is the following.

\begin{prop}\label{loc_sing}Let $[F]\in M$ be a polystable sheaf as in the statement of Proposition \ref{cones}. Suppose that $D_1\cdot D_2\ge 3$.
If the polystable sheaf $[F]$ lies in $P$ then, locally around $[F]$, the relative Prym variety $P$  is isomorphic to its tangent cone $C_{[F]}(P)$.
\end{prop}

{\sl Proof.}
First of all, recall from formula (\ref{dim_ext_mist}) that $\dim W=D_1\cdot D_2$.

Let $Z$ denote the irreducible component of the singular locus of $M$ containing $[F]$. Then  a neighbourhood $\mc Z$ of $[F_1\oplus F_2]$ in $Z$ is  isomorphic to $\mc M_1\times \mc M_2$ where, for $i=1,2$,  $\mc M_i\subset M_i=M_{v_i, H}$ is a  suitable analytic neighborhood of $[F_i]$  and $v_i=v(F_i)$.
As in Proposition 4.2 in \cite{Lehn-Sorger06}, one can show that the Kuranishi map vanishes identically on $\Ext^1(F_1, F_1)\oplus \Ext^1(F_2, F_2)$ so that under the identification $(\mc M, [F])\cong (\kappa^{-1}(0)//G, \,\,0)$, the pointed space $(\mc Z, [F])$ is  identified with $(\Ext^1(F_1, F_1)\oplus \Ext^1(F_2, F_2), 0)$.
Consider the natural projection
\[
 \Ext^1(F,F) \to \Ext^1(F_1, F_1)\oplus \Ext^1(F_2, F_2).
 \]
Since this projection is $G$-equivariant, and the action of $G$ on $k^{-1}(0)$ is induced by the linear action on $ \Ext^1(F,F)$, the restriction
\[
\kappa^{-1}(0) \to\Ext^1(F_1, F_1)\oplus \Ext^1(F_2, F_2),
\]
is also $G$-equivariant.
It follows that there is an induced morphism (recall that the $G$-action on the pure part of $\Ext^1(F,F)$ is trivial),
\[ p: (\mc M, [F]) \to (\mc Z, [F])\cong (\Ext^1(F_1, F_1)\oplus \Ext^1(F_2, F_2), 0).
\]
The inclusion $(\mc Z, [F]) \subset (\mc M, [F])$ defines a section of $ p$, and that the fiber of $ p$ over $[F]$ is identified with $(\kappa^{-1}(0)\cap \Ext^1(F_1, F_2)\oplus \Ext^1(F_1, F_2)) //G $. 

Now let $\Sigma$ denote the irreducible component of the singular locus of $P$ containing $[F]$. Then locally around $[F]$,  $\Sigma$ is isomorphic to $\mathcal P_1\times \mathcal P_2\subset \mathcal M_1\times \mathcal M_2$ where, for $i=1,2$,   $\mc P_i$ is a suitable neighborhood of $[F_i]$  in the   relative Prym variety  $P_i\subset M_i$. From Remark \ref{intr_cones} we deduce that the tangent cone of $P$ in $[F]$ is locally isomorphic to
\be\label{tan_cone_p}
\mc P_1\times \mc P_2\times \End_1^s(W).
\ee
The projection $\mc M \to \mc M_1 \times \mc M_2$ induces a morphism $ q: \mc P \to \mc M_1\times \mc M_2$, and since the morphism induced by $ p$ at the level of tangent cones is equivariant with respect to the induced action of $\tau$, the image of $q$ in $\mc M_1\times\mc M_2$ has the same dimension as, and therefore is equal to, $\mc P_1\times \mc P_2$. Thus there is a morphism $ q:\mc P \to \mc P_1 \times \mc P_2$. At the level of tangent cones, $ q$ is just the projection onto the first two factors in (\ref{tan_cone_p}). In particular, the fibration is flat and is thus a deformation of the central fiber $ q^{-1}(0)$. Since the tangent cone to  $ q^{-1}(0)$ is isomorphic to $\End_1^s(W)$, we can apply the theorem of Grauert   (cf. \cite{Grauert} and \cite{Camacho-Movasati}, Theorem 4.4) and conclude that $ q^{-1}(0)\cong \End_1^s(W)$. Moreover, since $\End_1^s(W)$ is isomorphic to the cone over the degree two Veronese embedding of $\PP W$, it is rigid as soon as $\dim W\geq 3$. It follows, as in the previous case, that locally around $[F]$,
\[
\mc P\cong\mc  P_1\times \mc P_2\times \End_1^s(W).
\]
\hfill $\square$

The main goal of  this section is to prove that when $|D|$ is not a hyperelliptic system the relative Prym variety  does not admit a symplectic resolution.

\begin{theorem} \label{no symplectic resolution} Let $T$ be a general Enriques surface and $f: S\to T$ its universal covering. Let $|C|$ be a non-hyperelliptic linear system of genus $g$ on $T$.
Set $D=f^{-1}(C)$ and $v=(0, [D], -h+1)$, $h=2g+1=g(D)$. 
The singular variety $P_{v,H}=\Prym_{v, H}(\D/\CC)$ defined in Section \ref{rel_prym}, does not admit a symplectic resolution.
\end{theorem}

\begin{proof}
Using Corollary \ref{two_curves}, we can find a curve in $|D|$ which is the union of two smooth $\iota$-invariant curves $D_1$ and $D_2$ meeting transversally in $2\nu$ points. Recall that, from our assumption on $C$ and from the same corollary,  it follows that $\nu \ge 2$.
We first examine the case $H=D$.
By Corollary \ref{existence polystable}, we can find a $D$-polystable sheaf
\be \label{polystable}
F=F_1\oplus F_2,
\ee
with $\supp(F_i)=D_i$ for $i=1,2$. Notice that  we can choose $F_1$ and $F_2$ to be $\tau$-invariant. We claim that we can moreover choose them so that $[F]$ belongs to the identity component $P_{v,D}=\Fix^0(\tau)$.
Indeed, by Proposition \ref{polyst} {\it every}  maximal dimension component  of $\Fix(\tau)$  contains a polystable sheaf as in (\ref{polystable}).

We claim that  the singularity of $P$ at $[F]$ is $\Q$-factorial and terminal  (but not smooth). This  proves that $P$ does not admit any symplectic resolution.

By Proposition \ref{loc_sing} above it follows that locally around $[F]$, the relative Prym variety $P_{v,D}$ is isomorphic to the product of an affine space times  the space of symmetric endomorphisms of $W=\Ext^1(F_2,F_1)$ of rank $\le 1$.  Moreover,
\[
\dim W=2\nu,\,\,\,\, \text{ with} \,\,\, \nu \ge 2.
\]
To prove the theorem, it is thus sufficient to prove the claim for  $\End_1^s(W)$, in the case when  $\dim W\geq 4$. The fact that $\End_1^s(W)$ is $\Q$-factorial follows from the fact that it os isomorphic to the quotient of $W$ by $\Z/2\Z$ acting by multiplication by $-1$. As for the statement on the type of singularity, what we need is contained in Example 1.5 (ii) of \cite{Reid}. Following the notation of Reid, we have 
\[
n=2\nu, \,\,\,\text{and}\,\,\, k=2,
\]

so that $b=\nu$ and $a=1$.
Also notice  that the blow up $\wh \End_1^s(W)$ of $\End_1^s(W)$ at the origin is smooth. From Reid's computation, it follows that the discrepancy of the exceptional divisor of the blow up is $\nu-1$.
As $\End_1^s(W)$ is $\Q$-factorial,  any resolution factors via the blow up and therefore its discrepancy is bounded from below by the discrepancy of the blow up.
In conclusion,  the singularity is terminal if $\nu >1$. Moreover, when $\nu=1$ the singularity is canonical but not terminal.

By a theorem of Flenner \cite{Flenner}, which we can use since by Proposition \ref{codim_sing} below the codimension of the singular locus of $P_{v,D}$ is greater or equal to two,  we know that the symplectic form extends to a holomorphic form on any smooth resolution of $P_{v,D}$.
However, from what we proved above we also know that the top exterior power of this holomorphic form vanishes along some exceptional divisors in any resolution.
We may conclude that no resolution of $P_{v,D}$ admits a non-degenerate symplectic form. 

We now consider the case of an arbitrary polarization $H$.
This case can be reduced to the case $H=D$, by claiming that
there always exists a point $[F']\in P_{v,H}$ and a birational morphism $\sigma: P_{v,H}\dasharrow P_{v,D}$ which is defined in a neighborhood of $[F']$ and such that $\sigma([F'])=[F]$. 
To prove the claim we consider the wall and chamber decomposition of $\Amp(S)$ determined by $v$. Let
$\{U_\alpha\}_{\alpha\in A}$ be the set of open chambers whose closure
contains $D$, let $U$ be the interior of $\cup_{\alpha\in A}\ov U_\alpha$.
For every $H'$ belonging to $U$ there is a regular map $\sigma: P_{v,H'}\to P_{v,D}$  (typically one expects, for this map, a positive dimensional fiber over the point $[F]$). 
Suppose $H$ does not belong to $U$.
For each $\alpha\in A$ we  choose a point  $H_\alpha\in U_\alpha\subset U$
and a point $[F_\alpha]\in P_{v, H_\alpha}$ mapping to $[F]$ under a rational map $\sigma_\alpha: P_{v,H_\alpha}\dasharrow P_{v,D}$.  We may then choose a  path from $H$ to $H_\alpha$
which only crosses walls not passing through $[F]$. Each of this crossing correspond to a birational map which is defined
in a neighborhood of $[F_\alpha]$.

\end{proof}

Notice that from the proof of Theorem \ref{no symplectic resolution}, it follows that if $\nu=1$, then the singularity is canonical and moreover the blow up of the singularity is a (local) crepant resolution. In particular, the pullback of the symplectic form is non-degenerate along that divisor.
This case will be studied in the next section.

\begin{rem} {\rm
It was proved by Kaledin \cite{Kaledin06} and by Namikawa \cite{Namikawa2}, that if $w: X \to Y$ is a birational projective morphism from a smooth projective symplectic $n$-fold $X$ to a normal variety $Y$, then $w$ is a semi-small map, i.e. if $Y_i$ denotes the set of points  $y \in Y$ such that $\dim w^{-1}(y)=i$, then $\dim Y_i \le n-2i$. This, in particular, implies that if $Y$ is $\Q$-factorial then $\codim \operatorname{Sing} (Y) \le 2$. From this observation it follows that if we knew by other means that $P$ (notation as above) were normal and $\Q$-factorial at a general singular point, then we could conclude that $P$ has no symplectic resolution.
}\end{rem}


\section{The hyperelliptic case}\label{hyp}

In this section, we will prove that the degree zero relative Prym variety of a hyperelliptic linear system is birational to a hyperk\"ahler manifold, and we will highlight the cases in which it is a smooth compact hyperk\"ahler manifold and the ones in which it admits a symplectic resolution. The hyperk\"ahler manifolds arising from degree zero relative Prym varieties are all of K3$^{[n]}$ type.
As in the preceding sections we will use the following notation (see (\ref{notation_prym}))
$$
M=M_{v,H}\,,\qquad \tau:M \dashrightarrow M\,,\qquad P=P_{v,H}
$$
where  $v=(0,[D], -h+1)$, and where now $|D|$ is a hyperelliptic linear system which is the pull-back, via $f: S\to T$, of a hyperelliptic linear system $|C|=|ne_1+e_2|$ on $T$
(see the Appendix \ref{appendix} for the relevant definitions and notation).

The first result we want to prove is that, in the hyperelliptic case, the Prym involution $\tau$ on $M$ comes from a bona fide
involution on $S$.

\begin{prop}\label{involution_k}
Let $f:S\to T$ and $\iota:S\to S$ be as in (\ref{univ_cover}) and (\ref{iota}).
Let $|C|=|ne_1+e_2|$ be a hyperelliptic linear system on $T$  and let $|D|=|f^*(C{)}|$. Set $v=(0,[D], -h+1)$, $h=g(D)$. Then there exists a symplectic involution
\[
k: S \to S,
\]
such that, for any $\iota^*$-invariant polarization $H$, the birational involution
\[
\tau: M\dashrightarrow M
\]
defined in Section \ref{rel_prym} concides with the birational involution
\[
\aligned
k^*:  M& \dashrightarrow M\\
F & \mapsto k^* F.
\endaligned
\]
\end{prop}

The proof of this proposition  consists in defining the involution $k$. In the  Appendix \ref{appendix}  we  recall the  notation and some basic results regarding the geometry of hyperelliptic linear systems on Enriques and K3 surfaces.
Let us then consider a hyperelliptic linear system
\[
|C|:=|ne_1+e_2|,
\]
where $e_1$ and $e_2$ are two primitive elliptic curves such that $e_1\cdot e_2=1$. The linear system has two simple base points (Proposition 4.5.1 of \cite{Cossec-Dolgachev}) and defines a degree two map of $T$ onto a degree $n-1$ surface in $\P^n$.
The two base points are
\be \label{base points odd}
e_1\cap e'_2 , \,\,\, \text{and }\,\, e'_1\cap e_2, \,\,\,\, \text{if} \,\, n \,\,\text{is odd,}
\ee
or
\be \label{base points even}
e_1\cap e_2 , \,\,\, \text{and }\,\, e'_1\cap e'_2, \,\,\,\, \text{if} \,\, n \,\,\text{is even.}
\ee
Following the notation in the Appendix \ref{appendix}  we have
\[
E_i=f^{-1}(e_i),  \,\, \text{and }\,\, E'_i=f^{-1}(e'_i), \,\,\,\, i=1,2.
\]

Set $h=2n+1$.
The genus $h$ linear system $|D|=|nE_1+E_2|$ is also hyperelliptic,  with the $g^1_2$ cut out by the elliptic pencil $|E_1|$. 
As in (\ref{phi_d}) consider the morphism
\[
\varphi=\varphi_D: S\to R\subset \PP^h
\]
attached to the linear system $|D|$. It is a degree two morphism onto a rational normal scroll of degree $h-1$.
Let us denote by
\[
\ell: S \to S,
\]
the anti-symplectic involution defined by $\varphi$.  Notice that any curve in $|D|$ is $\ell$-invariant, and that $\ell$ induces the hyperelliptic involution on every smooth member of the linear system.
In particular, we notice the well  known fact that if $D$ is smooth, $\ell^*$ acts as $-\id$ on $\Pic^0(D)$

Finally, observe that the sub-linear system 
\be\label{lin_syst_w}
W=f^*|C| \subset |D|
\ee 
has $4$ simple base points, that are the inverse image of (\ref{base points odd}) or of (\ref{base points even}),
By the same reasoning, the sub-linear system $W'=f^*|C'| \subset |D|$ has $4$ simple base points
that are the inverse image of (\ref{base points even}) or of (\ref{base points odd}) respectively.
Denote by $\{w_1, \dots, w_4\}$ the base points of $W$ and by $\{w'_1, \dots, w'_4\}$ the base points of $W'$.

\begin{lem} \label{involutions commute}
The two involutions $\ell$ and $\iota$ commute and their composition
\be \label{tau}
k=\iota \circ \ell,
\ee
is a symplectic involution with eight fixed points:
\be \label{basepoints}
\{w_1, \dots, w_4, w'_1, \dots, w'_4\}.
\ee
\end{lem}
\begin{proof}

It is sufficient to prove that for any smooth  $\iota$-invariant curve $\Gamma$ in $|D|$, the identity  $k^2_{|\Gamma}=\id_\Gamma$ holds. But this is clear, since, $\ell^*=-\id$ on $\Pic^0(\Gamma)$ and hence $(k^2)^*$ is the identity on $\Pic^0(\Gamma)$. 

Since $k$ is the composition of two anti-symplectic involutions, it is symplectic. As such, $k$ has eight fixed points which we readily describe.  For simplicity, let us focus our attention on the two points $\{p, q\}=E'_1\cap E_2$. By construction, we know that $\iota(p)=q$. Notice, however, that we can choose a smooth curve $D$ in the linear system $|nE_1+E_2|$ passing through those two points so that $D\cap E_1=\{p, q\}$. Since $|E_1|$ induces the $g^1_2$ on $D$, it follows that $\ell(p)=q$ and thus
\[
k(p)=p, \,\,\text{and } \,\, k(q)=q.
\]
We can now argue in the same way for the remaining points of (\ref{basepoints}).
\end{proof}

This Lemma, together the fact that $\ell^*=-1$ on $\Pic^0(D)$ for a smooth curve $D$, shows that $\tau=k^*$ on a dense open subset of $M$, thus proving  Proposition \ref{involution_k}.

\begin{rem}
If we choose $H=D$, so that $\tau$ is a morphism, then so is $k^*$. In fact, in this case, the linear system $|D|$ is $k^*$-invariant, and hence pulling back via $k$ respects $D$-stability.
\end{rem}

Recall from Section \ref{rel_prym},  Proposition \ref{never_preserve} that
 it is not possible to choose $H$  such that $M=M_{v, H}$ is smooth and the assignment $F\mapsto \Shext^1(F, {\mc O}(-D))$ respects stability.
It is therefore natural to ask wether there are choices of $H$ such that  $M_{v, H}(S)$ is smooth and at the same time $k^*$ is a regular morphism.  Equivalently, we ask whether there exists a $k^*$-invariant ample class which is also $v$-generic.

To describe the $v$-generic polarizations, we need to identify the walls.

\begin{lem}
The equations defining the walls relative to a primitive Mukai vector $v=(0, D, \chi)$, with $D=n E_1 + E_2$, are of the form
\be \label{walls of D}
\ff{s E_1\cdot x+\epsilon E_2\cdot x}{n E_1\cdot x+ E_2\cdot x} \,\, \chi =m,
\ee
with $m$ ranging in a finite sets of integers, $\epsilon \in \{0, 1\}$ and  $s=0, \dots, n-1$.
\end{lem}
\begin{proof} Since by assumption $D$ is primitive in $\NS(S)$, and since $\chi \neq 0$, the finiteness of the number of walls is proved in Section 1.4 of Yoshioka's paper \cite{Yoshioka}. Each sub-curve of $|D|$ belongs to a linear subsystem of type $|s E_1+ \epsilon E_2|$,
with $s$ ranging from $0$ to $n$, and $\epsilon \in \{0, 1\}$.  Up to passing to a residual series, we may assume that $\epsilon=1$. Clearly, equations  (\ref{walls of D}) are satisfied by   strictly semi-stable sheaves $F$ with $[c_1(F)]=[s E_1+  E_2]$ and $\chi(F)=m$. Notice, however, that the above equations are also sufficient for the existence of strictly semistable sheaves.
Indeed, consider a smooth curve $\Gamma \in |s E_1+ E_2|$ and a smooth curve $\overline \Gamma \in |E_1|$. If equation (\ref{walls of D}) holds for some ample class $H=x$ and some integer $m$, then we can choose a torsion free $H$-stable sheaf $F_\Gamma$ on $\Gamma$ with Euler characteristic equal to $m$, and an $H$-semi-stable sheaf $F_{\overline \Gamma}$  of rank $(n-s)$ on $\overline \Gamma$ such that $\chi(F_{\overline \Gamma})=\chi-m$. The sheaf $F=F_\Gamma \oplus F_{\overline \Gamma}$ is then strictly $H$-semi-stable. Since the above are all the possible sub curves of $|D|$, there are no other walls.\\
\end{proof}

We now turn to the question regarding the  $k^*$-invariance of the polarization.
With this in mind, we describe the action of $k^*$ on $\NS(S)$. Notice that since the involutions $\iota$, $\ell$ and $k$ commute, the induced actions on the N\'eron-Severi group are compatible.
For a lattice $\Lambda$ and an involution $\epsilon$ acting on $\Lambda$, we denote by $\Lambda^\varepsilon$ and by $\Lambda^{-\varepsilon}$ the invariant and anti-invariant sub-lattices, respectively. Note that $\Lambda^\varepsilon$ and $\Lambda^{-\varepsilon}$ are primitive in $\Lambda$.

Recall that $\NS(T)\cong U\oplus E_8(-1)$ and that $f^* \NS(T)$ is a primitive  $10$-dimensional sub-lattice of $H^2(S, \Z)$. By Proposition \ref{hyperelliptic k3},  it follows that $\NS(R)\cong U$ and that $\varphi^* \NS(R)=\langle E_1, E_2 \rangle$ is also primitive in $H^2(S, \Z)$. 
For simplicity we will indicate by $\NS(T)$ and by $\NS(R)$ their respective pull-backs in $\NS(S)$.

\begin{lem}\label{neron_severi}$\NS(S)^{k^*}=\langle E_1, E_2 \rangle\oplus (\NS(T)^\perp)^{k^*} $,  so that if  $T$ is a general Enriques surface then $\NS(S)^{k^*}=\langle E_1, E_2 \rangle$.
\end{lem}

{\sl Proof.} Since $\NS(S)^{\ell^*}=\NS(R{)}=\langle E_1, E_2\rangle\subset \NS(T)$, we have $\langle E_1, E_2 \rangle=\NS(T)^{\ell^*}\cong  \NS(S)^{\ell^*}$.
Moreover, by \cite{Nikulin76}, \cite{Morrison84}, \cite{vanGeemenSarti}
\[
H^2(S, \Z)^{k^*}\cong U^{\oplus 3} \oplus E_8(-2), \,\,\, \text{and} \,\,  H^2(S, \Z)^{-k^*}\cong E_8(-2).
\]
Since $k^*$ preserves the $(2,0)$ part of the Hodge decomposition we have 
$$
H^2(S, \Z)^{-k^*} \subset H^{1,1}(S)\cap H^2(S,\Z)=\NS(S)\,.
$$
In particular, since $\langle E_1, E_2 \rangle=\NS(T)^{\ell^*}=\NS(T)^{k^*}$, we have
\[
\NS(S)^{k^*}=\langle E_1, E_2 \rangle\oplus (\NS(T)^\perp)^{k^*}. 
\]
It follows that if the K3 surface satisfies $\NS(S)=\NS(T)$, then $\NS(S)^{k^*}$ is spanned by the classes of $E_1$ and $E_2$. This certainly happens if $T$ is general (but if the Picard number of $S$ is strictly greater than $10$, then there may be $k^*$ invariant classes that do not come from $T$).
\hfill $\square$

\begin{cor} Let $T$ be an Enriques surface.
Let $\mathcal W_v\subset \Amp(S)$ be the union of all the walls of $v=(0, D, \chi)$. Then
\[
\Amp (S) \cap \NS(S)^{k^*} \setminus \mathcal W_v \cap \NS(S)^{k^*},
\]
is non empty. 
\end{cor}
\begin{proof}
Since there are finitely many walls, it is sufficient to check that none of the equations (\ref{walls of D}) vanish identically on $\NS(S)^{k^*}=\langle E_1, E_2 \rangle$. To this aim, consider $H \in \Amp (S) \cap \langle E_1, E_2 \rangle$. Then $H=aE_1+bE_2$ for some positive integers $a$ and $b$, so that the restrictions of the equations to $\Amp (S) \cap \NS(R)$ are
\be \label{wall restricted to NS(R)}
\ff{bk+\epsilon a}{bn+a} \chi=m,
\ee
and these are not identically zero.
\end{proof}

\begin{prop} \label{H k-invariant and v-generic}
Keeping the notation of Theorem \ref{involution_k}, we can always find an ample divisor $H$ which is $k^*$-invariant and $v$-generic. In other words, there exists a polarization $H$ such that $M_{v, H}$ is smooth and such that
\[
k^*:  M_{v, H} \longrightarrow M_{v, H},
\]
is a regular involution. In particular, if $H\in \NS(R) \setminus \mathcal W_v \cap \NS(R)$, the a priori birational involution $
\tau: M_{v, H} \dasharrow M_{v, H}$, extends to a regular morphism and the relative Prym variety $\Prym_{v, H}(\mc D/\mc C)$ is smooth and symplectic.
\end{prop}
\begin{proof}
It is sufficient to consider $H \in \Amp(S)\cap \NS(S)^{k^*} \setminus  \mathcal W_v \cap \NS(S)^{k^*} $ which is non empty by the above Corollary.
\end{proof}

Since changing the polarization $H$ does not change the birational class of the relative Prym variety and the above proposition ensures the existence of smooth Prym varieties, we can sum up the results we obtained so far in the following

\begin{theorem} \label{smooth symplectic compactification}
Let $T$ be an Enriques surface and $f: S\to T$ the universal covering. Let $|C|=|ne_1+e_2|$ be a hyperelliptic linear system on $T$  and let $|D|=|f^*(C{)}|$. Set $v=(0,[D], -h+1)$, $h=g(D)$.
For any polarization $H$ on $S$, the relative Prym variety $P_{v, H}=\ov{\Fix^0(\tau)}=\ov{\Fix^0(k^*)}$  is birational to a projective hyperk\"ahler manifold.
\end{theorem}

\begin{cor}
The singular symplectic variety $P_{v, D}$ admits a symplectic resolution.
\end{cor}
\begin{proof}
First, recall that $D$ is not $v$-generic. In fact, since the genus $h$ of $D$ is odd, $\chi=-h+1$ is even, and hence $D$ lies on the wall with equation 
\[
\frac{E_2\cdot x}{n E_1\cdot x+ E_2\cdot x}\,\, \chi =2.
\]
It is possible to choose an ample $H$ as in Proposition \ref{H k-invariant and v-generic} and such that, moreover, it lies in a $v$-chamber adjacent to the wall (or the intersection of walls) where $D$ lies.
Under these assumptions, there is a natural projective birational morphism
\[
\varepsilon:M_{v, H} \to M_{v, D},
\]
which is a resolution of the singularities of $M_{v, D}$. Since $\varepsilon$ commutes with both $k^*:M_{v, H}\to M_{v, H}$ and $\tau: M_{v, D}\to M_{v, D}$, there is an induced proper morphism
\[
\varepsilon:P_{v, H} \to P_{v, D}.
\]
which is still birational. Since $P_{v, H}$ is smooth and symplectic, it is a symplectic resolution of the singularities of $P_{v, D}$.
\end{proof}

The natural question is now to determine the deformation class of these smooth relative Pryms. As one can expect they are birational, and thus by what Huybrechts proves in \cite{Huybrechts2} deformation equivalent, to some moduli spaces of sheaves on the minimal resolution of the quotient of $S$ by $k$.

To fix notation, let
\[
\rho: S \to \Sb:=S/\langle k \rangle,
\]
be the quotient morphism. Then $\Sb$ is a singular K3 surface with $8$ rational double points and we let
\[
\eta: \Sh \to \Sb,
\]
be its minimal resolution. It is well known that $\Sh$ is a K3.
Observe that a divisor in $|D|$ is $k$-invariant if and only if it is $\iota$-invariant. In particular, any  $D$ in $W$ is $k$-invariant so that if we set $\Db=D/k$  there is an obvious bijection
$W\cong|\Db|$.
The general curve $D$ in $W$ does not contain the points $\{w'_1, \dots, w'_4\}$ so that for $D\in W$,  the double cover  $D\to \Db$ ramifies only in $\{w_1, \dots, w_4\}$ and
\[
g(\Db)=\ff{h-1}{2}=g-1.
\]
Moreover, if $D\in W$ is smooth, then so is $\Db$. In this case, the proper transform
\[
\Dh:=\eta^{-1}_*(\Db)\subset \Sh,
\]
is also smooth and isomorphic to $\Db$.  Consider now the following commutative diagram,
\[
\xymatrix{
Z  \ar@(ul,dl)[]|{\wh k} \ar[r]^{\rh} \ar[d]_{\eh} & \Sh \ar[d]^\eta \\
S\ar@(ul,dl)[]|{ k} \ar[r]_\rho & \Sb=S/k \,,
} 
\]
where $Z \to S$ is the blow up of the surface at the $8$ fixed points of the involution, and $Z \to \Sh$ is the double cover ramified along the eight exceptional curves of $\eta$.
We denote by 
$$
\wh k : Z \to Z
$$ 
the lift of $k$ to $Z$ so that $\wh k$ is an involution on $Z$ whose fixed locus is the union of the exceptional divisors.
Let $R_1, \dots, R_4$ be the exceptional divisors mapping to $\rho(w_1), \dots, \rho(w_4)$,  and  $R'_1, \dots, R'_4$ the exceptional divisors mapping to $\rho(w'_1), \dots, \rho(w'_4)$. Since the general curve in $|D|$ does not pass through $w_1', \dots, w_4'$ and is smooth,
\be \label{diramazione}
\wh D \cdot R_i=1,
\ee
and $\wh D\cdot R_i'=0$.

The general curve $\Gamma$ in $\wh \rho^{\,*} |\wh D|$
is a smooth double cover of a curve in $|\wh D|$,
and, via $\eh$, maps isomorphically to its image in $S$. 
Indeed, $\eh$ induces an isomorphism
\be \label{Gamma e W}
|\Gamma| \supset \wh \rho^{\,*}|\wh D|\cong W \subset |D|.
\ee

\begin{theorem} Set $w=(0, \wh D, -g+2)$, and let $H$ and $\wh H$ be two ample line bundles on $S$ and $\Sh$ respectively.
There is a rational map
\be \label{rational psi}
\begin{aligned}
\psi:M_{w, \wh H}(\Sh\,) &\dashrightarrow M_{v,H}(S),\\
F &\mapsto \wh \eta_* \wh \rho^{\,*} F.
\end{aligned}
\ee
defined on the open set of $M_{w, \wh H}(\Sh\,)$  parametrizing sheaves supported on irreducible curves. This map factors via the inclusion $P_{v,H}\subset M_{v,H}(S)$, and the induced map
\be \label{rational phi}
\phi:M_{w, \wh H}(\Sh\,) \dashrightarrow P_{v,H},
\ee
is birational.
\end{theorem}

\begin{proof}
For our purposes, it is enough to restrict our attention to the open subset of $M_{w, \wh H}(\Sh)$ parametrizing sheaves with smooth support. 
Let $\F$ a  family of pure sheaves of dimension one on $S$ with Mukai vector $w$, supported on smooth curves, and parametrized by a scheme $B$. Then $\wh \eta_{B*} {\wh\rho_B}^{\,*} \F$ is a flat family of $H$-stable sheaves.
Clearly, $\wh \rho_B^{\,*}\F$ is flat over $B$ and by formula (\ref{diramazione})  the support $\Gamma_b$ of $\rhs \F_b$ is smooth.
If $\E$ is any flat family, parametrized by $B$,  of pure dimension one sheaves on $Z$, with support on smooth curves belonging to $|\Gamma|$ then $\wh \eta_{B*} \E$ is a flat family of $H$-stable sheaves with support in $W$. 
This defines the rational map (\ref{rational psi}).
Since $\wh k$ is a lift of $k$, the two involutions coincide where $\eh$ is an isomorphism. Let $[F]$ be a point in $M_{w, \wh H}(\Sh\,)$, since $ \rhs F$ is $\wh k^*$-invariant, it follows that 
\[
\wh \eta_* \wh \rho^{\,*} F \in \Fix(k^*).
\]
Hence  $\psi$ factors through the inclusion $P_{v, H}\subset M_{v,H}(S)$.
The last assertion to prove is that the induced map $\phi:M_{w, \wt H}(\St) \dashrightarrow P_{v,H}$ is birational.
As above, we assume that $\supp(F)$ is smooth. First, we show that  $\phi$ is a local isomorphism at $[F]$. We claim that the induced map on tangent space
\be \label{tangent map}
d \psi: \Ext^1_{\wh S}(F, F) \to \Ext^1_Z(\rhs F ,\rhs F ) \cong \Ext^1_S(\eh_* \rhs F ,\eh_* \rhs F ).
\ee
is injective. In fact, given a non trivial extension $0 \to   F \to G \to   F \to 0
$, we can pull it back to $Z$ to obtain a short exact sequence
\be \label{short exact sequence}
0 \to \rhs F \to \rhs G \to  \rhs F \to 0.
\ee
If this sequence were split, the same would be true for
\[
0 \to  \rh_*  \rhs F \to \rh_*  \rhs G \to  \rh_*  \rhs F \to 0.
\]
However, this sequence is the direct sum of (\ref{short exact sequence}) and of
\[
0 \to  \rhs F\otimes L \to \rhs G \otimes L \to  \rhs F\otimes L \to 0,
\]
where\footnote{Recall that $\sum R_i +\sum R_i'$ is divisible by two in $\NS(\wt S)$.}
\[
L:= \ff{1}{2} {\mc O}_{\wt S} (-\sum R_i -\sum R_i').
\]
Since these two exact sequence are non split by assumption, we get a contradiction. Hence, the induced map (\ref{tangent map}) is injective and  $\phi$ is a local isomorphism. 
To end the proof of the theorem, we just need to prove that the degree of $\phi$ is one.
It is enough to prove that if $F_1$ and $F_2$ are two sheaves on $\Sh$ with Mukai vector $w$, then
\[
\rhs F_1 \cong \rhs F_2, \,\,\,\, \text{if and only if} \,\,\,\,F_1 \cong F_2.
\]
This follows from the projection formula. In fact, if $\rhs F_1 \cong \rhs F_2$, then
\[
F_1\oplus (F_1 \otimes L)\cong \rh_* \rhs F_1 \cong \rh_* \rhs F_2\cong F_2\oplus (F_2 \otimes L).
\]
Since, for $i=1,2$, both  $F_i$ and  $F_i \otimes L$ are stable and since $\deg F_i \otimes L \neq \deg F_j$, $i, j=1,2$ we must have an isomorphism
$F_1 \cong F_2$.

\end{proof}

\begin{cor}\label{hyp_K3_type} Let $D=f^*C$ and $v=(0, [D], -h+1)$. If $|C|$ is a hyperelliptic linear system, and  $H$ is $v$-generic and $k^*$-invariant, the symplectic variety $P_{v,H}$ is an irreducible holomorphic symplectic manifold of type $\Hilb^{g-1}(K3)$.
\end{cor}

\begin{cor} \label{hyp_K3_type2}
Let $D=f^*C$, $v=(0, [D], -h+1)$ and let $H$ be a non $v$-generic polarization. If $|C|$ is a hyperelliptic linear system, then any resolution of $P_{v,H}$ is an irreducible holomorphic symplectic manifold that is of type $\Hilb^{g-1}(K3)$.
\end{cor}


\subsubsection{The surface case}\label{surface_case} 
In \cite{Sacca13}, the case of the relative Prym variety associated to a genus $2$ linear system $|C|=|e_1+e_2|$ on an Enriques surface $T$ is studied in greater detail. Here we report the  results of that analysis.

 Recall that in this case $|C|\cong \P^1$, and that the linear system has two simple base points. If we assume the pair $(T,C)$ to be general  (cf. \cite{Sacca12}) the linear system $|C|=\P^1$ has exactly $16$ irreducible curves with one single node, and $2$ reducible curves $e_1+e_2$ and $e_1'+e_2'$ that are the union of two smooth elliptic curves  meeting transversally in one point.

As usual we look at the 2-sheeted K3 cover $f: S\to T$ and we set $|D|=|f^*C|$, and $v=(0,[D], -2)$. We also choose   a $v$-generic polarization $H$.
We consider the involution 
$\tau=\tau_{-D}: M_{v,H}\to M_{v,H}$ and 
our goal is to describe $P_{v,H}$. The following theorem holds. 

\begin{theorem}\label{prym_surface}
The relative Prym variety $P_{v,H}$ is an elliptic K3 surface whose singular fibers consist in $16$ irreducible curves with one node, and two fibers of type $\operatorname{I}_4$  (i.e. a closed chain of $4$ rational curves $R_1,\dots,R_4$ with $R_i\cdot R_{1+1}=1$, $i=1,\dots, 3$ and $R_4\cdot R_1=1$ ).
\end{theorem}
\begin{proof} The fact that there exist 16 irreducible rational nodal curves follows from Corollary \ref{rank one degeneration}. To show that there are two fibers of type $I_4$ we proceed as follows. 
From Proposition \ref{polyst}, applied to $P_{v,D}$, there are two fibers that are the union of two rational curves meeting at two points. These points of intersection correspond to polystable sheaves of type $F_1\oplus F_2$
and by the analysis carried out in Section \ref{singular}, the resulting singularities are of type $A_1$.
The proof follows from the fact that $P_{v,H}\to P_{v,D}$ is a resolution of singularities.
\end{proof}


\section{The fundamental group}
\label{fund_group}

Let $T$ be a general Enriques surface and $f: S\to T$ its universal covering.
Let $C\subset T$ be a smooth, primitive, curve of genus $g\geq 2$ and set $D=f^{-1}{(}C)$. Consider the Mukai vector $v=(0, [D], -h+1)$, with $h=g(D)=2g-1$.

From Section \ref{hyp}, we  know that if $|C|$ is hyperelliptic, then 
$P_{v,H}$ is either an  irreducible symplectic manifold, and thus  simply connected,  or else has a resolution 
which is one such. 
 We can thus restrict our attention to the non-hyperelliptic case. 
 
 \begin{theorem} \label{simple_conn} Let $|C|$ be a non-hyperelliptic system on a general Enriques surface. Let $P_{v,H}$ be the relative
 Prym variety associated to $|C|$ (cf. (\ref{notation_prym}) ).  Any resolution $\wt P_{v,H}$ of the singularities of $P_{v,H}$ is simply connected.
\end{theorem}

We set,
\be\label {p_fund}
M=M_{v,H}\,,\qquad P=P_{v,H}.
\ee
Consider the support morphism
\be\label{jac_fibr}
\pi: J=M\longrightarrow |D|.
\ee 
 and look at its  restriction to $P$ 
\[
 \eta: P\to |C|\cong\P^{g-1}
\]
Let $U'$ be the locus of irreducible curves in $|C|$, $Z=|C|\smallsetminus U'$, and set
\[
P'=\eta^{-1}(U). 
\]
Let $\gamma: \wt P \to P$ be any resolution of singularities. Since $P' $ is contained in the smooth locus of $P$, by \cite{Fulton-Lazarsfled} 0.7.B  (cf. also \cite{Kollar-Sha} Proposition 2.10), the natural morphism $\pi_1(P') \to \pi_1(\wt P)$ is surjective and hence to prove the theorem it is enough to prove that
\[
\pi_1(P')=\{1\}. 
\]
Notice that the same reasoning applies to show the simple connectivity of the normalization of $P$.

We will deduce Theorem \ref{simple_conn} from the simple connectivity of $M=M_{v,D}$ and from Picard-Lefschetz theory. Similarly to what is done in \cite{Mark_Tik},  we will use a theorem of Leibman \cite{Leibman} which we state in a form directly suited to our needs.

\begin{theorem}[Leibman]\label{leib}
Let $p \colon E\to B$ be a surjective morphism of connected smooth manifolds. Assume $p$ has a section $s$. Let $W \subset B$ be a  closed submanifolds of real codimension at least two. Set $U = B \setminus W$ and $E_U = p^{-1}(U)$ and assume that $E_U \to U$ is a locally trivial fibration with fiber $F$. Consider the exact sequence
\begin{equation} \label{leibman_seq1}
1 \to \pi_1(F) \overset {j_*}{\to} \pi_1(E_U) \overset{s_*}{\leftrightarrows} \pi_1(U) \to 1.
\end{equation}
Set $H = \ker (\pi_1(U)\to \pi_1(B))$. 
Via $j_*$, consider $\pi_1(F)$ as a normal subgroup of $\pi_1(E_U)$ and 
Let $R = [\pi_1(F), H]$ be the commutator subgroup of $\pi_1(F)$ and $ H$ in $\pi_1(E_U)$. Then there is an exact sequence
\begin{equation}\label{leibman_seq2}
1\to R\to\pi_1(F)\to\pi_1(E)\overset{s_*}\leftrightarrows\pi_1(B)\to 1.
\end{equation}
\end{theorem}
The commutator subgroup  $R$ of the statement of the theorem, should be understood as generated by elements of type 
\begin{equation} \label{conj}
c^{-1}\tilde{\lambda}^{-1} c \tilde{\lambda},
\end{equation}
where $c \in \pi_1(F)$ and $\tilde{\lambda} = s_{*}(\lambda)$ is a lifting of $\lambda \in \pi_1(U)$ to $\pi_1(E_U)$.

Before proving Theorem \ref {simple_conn}, 
let us  apply right away  Liebman's theorem to the support morphism 
\be\label{jac_fibr}
\pi: J=M\longrightarrow |D|.
\ee
To be more precise, we let $\Delta_D\subset |D|$ be the discriminant locus, $V=|D|\smallsetminus \Delta_D$ the locus of smooth curves, $V'\supset V$ the locus of irreducible curves, and we apply it the restriction of $\pi$ to the open subset $V'$ where the rational section $s$ of (\ref{rational section}) is defined. Recall that the complement of this open subset has codimension greater or equal to two.
Set $W=V' \cap \Delta_D$.

In this case both $E=J_{V'}$ and $B=V'$ are simply connected and, by the above theorem, we get 
$R=\pi_1(J{(}D_0))$ where $D_0$ is a smooth curve in $|D|$.
To unravel what this means, we first observe that, given $\lambda\in \pi_1(U)$,
the element $\tilde{\lambda}^{-1} c \tilde{\lambda}$ is 
the result of applying  the Picard-Lefschetz transformation,
attached to the loop $\lambda$, to the cycle $c$:
\begin{equation}\label{PL}
\aligned
\operatorname{PL}:\, \pi_1(U, &u)\longrightarrow  \Aut(H_1(D_0, \Z))=\Aut(\pi_1(J(D_0, \Z)))\\
&[\lambda]\mapsto \{c\mapsto \tilde{\lambda}^{-1} c \tilde{\lambda}\}
\endaligned
\end{equation}
To visualize $\pi_1(V)$, take a generic two plane $\Sigma \subset |D|$ and
consider the discriminant curve $\Gamma=\Delta_D\cap E$. By a classical theorem of Zariski
$\pi_1(\Sigma \smallsetminus \Gamma)= \pi_1(V)$ (\cite[Theorem $4.1.17$]{Dimca}).
Generators for $\pi_1(\Sigma \smallsetminus \Gamma)$ can be obtained by fixing a smooth point $z$ on the discriminant curve $\Gamma$ and taking the boundary of a small one dimensional disk contained in $|D|$ and meeting $\Gamma$ only in $z$ and transversally there. The family of curves parametrized by this disk is a family of smooth curves acquiring a simple node.
Let $\alpha_\lambda$ be the vanishing cycle of this family. It is a classical result that  the Picard-Lefschetz homomorphism (\ref{PL}) is given by
(see, for instance, \cite[Section $X.9$]{GACII})
\be\label{PL_alpha}
PL_{\lambda}(c)= c + (c \cdot \alpha_\lambda) \alpha_\lambda.
\ee
Going back to (\ref{conj}) and using additive notation (since $\pi_1(J{(}C)) = H_1(C, \Z)$), we get 
\begin{equation} \label{pl4}
c^{-1}\widetilde{\lambda}^{-1} c \widetilde{\lambda} = -c + PL_{\lambda(c)}
 =(c \cdot \alpha_\lambda) \alpha_\lambda.
\end{equation}
Thus, the simple connectivity of $J$, i.e. the equality $R=\pi_1(J{(}D_0))$, simply means that $\pi_1(J{(}D_0))$
is generated by vanishing cycles, as expected.

\begin{proof} [Proof of Theorem \ref {simple_conn}] Recall that $U'\subset |C|$ is the locus of irreducible curves.
 We want to apply Theorem \ref{leib} to the morphism
 \[
\eta: P'\longrightarrow U'
\]
(throughout, when there is no confusion, we use the same symbol for a morphism and its restrictions). Recall that $\eta: P' \to U'$ has a section induced by $s: V' \to J_{V'}$.
Let $W'=U'\smallsetminus U=\Delta\cap U'$ be the discriminant locus of $U'$.
As usual, via $f: S\to T$, we consider $|C|$ as a linear subspace of $|D|$.
Pick a point $u \in U$ corresponding to an unramified two-sheeted cover 
\begin{equation*}
f \colon D_0 \to C_0,
\end{equation*}
where $C_0 \subset T$ is a smooth member of $|C|$ while $D_0=f^{-1}{(}C_0) \subset S$ is a smooth member of $|D|$. We also set
\begin{equation*}
P_0= \Prym(D_0/C_0).
\end{equation*}
In the present case the sequence \eqref{leibman_seq1} is given by
\begin{equation*}
0 \to \pi_1(P_0, 0) \to \pi_1(P_{U}, 0) \to \pi_1 (U, u) \to 0,
\end{equation*}
where $P_{U} = \eta^{-1}(U)$ is the restriction to the smooth locus.

 By Proposition \ref{red_cod_one} $Z$ is of codimension $\geq 2$ in $|C|$ and hence the complement  $U'$ is simply connected. Thus, to prove the simple connectivity of $P'$ it suffices to prove that
\begin{equation} \label{da_fare}
\pi_1(P_0, 0) = [\pi_1(P_0, 0), \pi_1 (U, u) ];
\end{equation}
It will be useful to identify the first homotopy group of $F_P$ with the $\iota$-anti-invariant subspace of $H^1(D, \Z)$:
\begin{equation*}
\pi_1(P_0, 0) = H^1(D_0, \Z)_{-}
\end{equation*}
To prove \eqref{da_fare} we must make explicit the conjugation action of $\pi_1 (U', u)$ on $\pi_1(P_0, 0)$. We have a commutative diagram
\begin{equation}\label{fiber_prym}
\xymatrix{
0 \ar[r] & \pi_1(P_0, 0) \ar[d] \ar[r]^{j_*}\ & \pi_1(P_{U}, 0) \ar[d] \ar[r] & \pi_1(U, u) \ar[d] \ar[r] \ar@/_1pc/[l]_{{s'}_*} & 0
\\
0 \ar[r] &\pi_1(J(D_0), 0) \ar[r] ^{j_*}& \pi_1(J_V, 0)\ar[r]& \pi_1(V,u) \ar[r] \ar@/_1pc/[l]_{s_*} & 0.
}
\end{equation}
where $J_V=\pi^{-1}(V)$.
Let us look at a simple closed loop $\gamma$ in $U$ going around one of the smooth branches of $W'$. First of all we want to determine the image of $[\gamma]$ in $\pi_1(V, u) $. Let $W'_\gamma$ be the local branch of $W'$ around which $\gamma$ goes.
A general point $p$ in $W'_\gamma$ corresponds to an irreducible curve $C_p$ on the Enriques surface $T$ having one node and no other singularities. It also corresponds to a $\iota$-invariant curve $D_p$ on the K3 surface having exactly two nodes $a$ and $b$ as singularities which, by Lemma \ref{prim_irr} is irreducible. These two nodes are exchanged by the involution; in fact $C_p = D_p/\iota$.
Smoothing the node $a$ or the node $b$ corresponds to moving away from $p$ on two smooth local branches of $W$ meeting transversally along $W'_\gamma$. These two branches are exchanged by the involution $\iota$. 

Making a section with a generic 2-dimensional plane $\Sigma$,  we may assume that, locally we have 
\[
W\cap \Sigma \underset{loc}=\{(x,y) \in \C^2 \mid |x|<\epsilon\,, |y|<\epsilon\,, \,\, xy = 0 \}
\] 
while $W'_\gamma$ is the origin.  We may think that the image of  $\gamma$ in $V$ is given by $\gamma(t) = (\epsilon_0e^{2\pi it}, \epsilon_0e^{2\pi it})$. In 
\[
V\cap \Sigma=\{(x,y) \in \C^2 \mid |x|<\epsilon\,, |y|<\epsilon\,, \,\, xy \neq 0 \}
\]
$\gamma$ is homotopic to the composition of one loop $\lambda$ going around the $x$-axis and 
one loop $\mu$ going around the $y$-axis. Since the two branches of $W$ meeting in $W'_\gamma$ are exchanged by the involution, we may as well assume that $\mu=\iota\lambda$. 

In conclusion, there is a system of generators $\{[\gamma_s]\}_{s \in K}$ for $\pi_1(U,u)$ such that, for each $s$, $\gamma_s$ is a simple closed loop  having the property that, under the inclusion $j \colon U \into V$, one has 
\begin{equation} \label{jstar}
j_{*}([\gamma_s]) = [\lambda_s][\iota\lambda_s]
\end{equation}
where $\lambda_s$ is a simple closed loop. We claim that the elements
  $\{\lambda_s, \iota \lambda_s\}_{s\in K}$ generate $\pi_1(U,u)$. To prove this first observe 
that, if $l$ is a line in $U'$ meeting $W'$ transversally, then there is a surjection $\pi_1(l\setminus l\cap W', u) \to \pi_1(U, u)$. Now move the line $l$ in $|D|$ to get a line $m$, very close to $l$, and meeting $W$ transversally. Set $l \cap W' = \{ x_1, \dots, x_N \}$ where $N  = \deg W'$. Then  we may set $m \cap W = \{ y_1, \dots, y_{2N} \}$. Moreover we may assume that, for $s = 1, \dots, N$, $y_{2s}$ and $y_{2s - 1}$ belong each to one of the two local branches of $W$ meeting in the branch of $W'$ to which $x_s$ belongs. The claim follows from observing that also $\pi_1(m \setminus m \cap W, u) \to \pi_1(V, u)$ is surjective.

Going back to diagram \eqref{fiber_prym} we may now identify the action of $[\gamma]$ on $\pi_1(P_0, 0)$ as the action of $[\lambda][\iota\lambda]$ on the $\iota$-anti-invariant subspace $H^1(D, \Z)_{-} \subset H^1(D, \Z) = \pi_1(J(C), 0)$. Let $\alpha$ be a vanishing cycle on $D_0$ such that (\ref{PL_alpha}) holds. Recalling that $(\alpha \cdot \iota \alpha) = 0$ we have, as in \eqref{pl4}
\begin{equation} \label{pl5}
c^{-1} (\widetilde{\lambda} \cdot \widetilde{\iota \lambda})^{-1} c (\widetilde{\lambda} \cdot \widetilde{\iota\lambda}) = -c + P_{\lambda} P_{\iota\lambda}(c)
= (c \cdot \iota \alpha) \iota \alpha + (c \cdot \alpha) \alpha =
 (c \cdot \alpha)(\alpha - \iota\alpha).
\end{equation}
Let now $\{ \lambda_s, \iota\lambda_s\}_{s \in K}$ be as in  (\ref{jstar}) and let $\alpha_s$ be the vanishing cycle on $D_0$ corresponding to $\lambda_s$. Since this is a set of generators for $\pi_1(V, u)$ and since $J$ is simply connected,  we may assume that $\{\alpha_s, \iota\alpha_s\}_{s\in K}$ generate $\pi_1(J(D_0), 0) = H_1(D_0, \Z)$. In conclusion $[\pi_1 (U, u), \pi_1(P_0, 0)]$ is generated by elements of the form
\begin{equation*}
(c \cdot \alpha_s)(\alpha_s - \iota \alpha_s)\,, \quad s \in K,
\end{equation*}
where $c$ runs in $H^1(D, \Z)_{-} = \pi_1(P_0, 0)$. Since $\{\alpha_s, \iota\alpha_s\}_{s\in K}$ generate $H_1(D_0, \Z)$ the set $\{\alpha_s - \iota \alpha_s \}_{s\in K}$ generates $H^1(D_0, \Z)_{-}$. Thus, in order to prove \eqref{da_fare} it suffices to prove that for each $s \in K$ there exists $c_s \in H^1(D, \Z)_{-}$ such that $(c_s \cdot \alpha_s) = 1$. For  this it suffices to find, for each $s \in K$, a simple closed loop $\beta_s$ on $D$ such that $((\beta_s - \iota \beta_s) \cdot \alpha_s) = 1$; we will find one such that $(\beta_s \cdot \alpha_s) = 1$ and $(\beta_s \cdot \iota \alpha_s) = 0$. Both $\alpha_s$ and $\iota \alpha_s$ are vanishing cycles and by construction there is a curve $C_0 \in |H|$, having exactly two nodes, resulting from the vanishing of $\alpha_s$ and $\iota\alpha_s$, and no other singularities than the two nodes. But  $C_0$ is irreducible and therefore $C \setminus \{ \alpha_s, \iota \alpha_s \}$ is connected and $\beta_s$ can be readily constructed.
\end{proof}


\section{Computation of $h^{2,0}$}\label{hilb}

From the last corollary of the preceding section we deduce that, in the hyperelliptic case, the $h^{2,0}$-number of any desingularization of the relative Prym variety $P_{v,H}$ is equal to 1.

We next examine the non-hyperelliptic case.
Fix a general Enriques surface $T$ with its universal cover $f:S\to T$. Fix
a non-hyperelliptic genus $g$ system $|C|$ on $T$ and let
$D=f^*{(}C)$,  $\chi=-h+1=-2g+2$ and  $v=(0, D, \chi)$. In this section we set
\[
P=P_{v, H}.
\]

\begin{theorem}Suppose that $g$ is odd. Let $\wt P$ be any desingularization of $P$. Then
 $h^{2,0}(\wt P) = 1$. 
\end{theorem}
\begin{proof} 
We first show that $h^{2,0}(\wt P) \leq 1$.
Following  an idea already used in  \cite{Mark_Tik}, we construct a dominant rational map
\begin{equation} \label{hilbsch}
\phi \colon \Hilb^{g-1}(S) \tto P
\end{equation}
Set $V = H^0(C, {\mc O}_S(D)^{\dual}$. As $S$ is un-nodal, the linear system $|D|$ is very ample (cf. Theorem 6.1, (iii) of \cite{Saint-Donat}), 
so that $S \subset \P V \cong \P^{2g-1}$. After choosing a linearization, the involution $\iota$ induces a decomposition $V = V_{+} \oplus V_{-}$ into $\pm 1$ eigenspaces. The Enriques surface $T$ is contained in $\P V_{-} \cong \P^{g-1}$.
We may think of the double cover $f \colon S \to T$ as obtained by projecting from $\P^{2g-1}$ to $\P^{g-1}$ with center the $(g-1)$ linear subspace $\Lambda = \P V_{+}$. Consider the open subset $U$ of $\Hilb^{g-1}(S)$ consisting $(g-1)$-tuples $\{ p_1, \dots, p_{g-1} \}$ of distinct points on $S$, such that 

a) 
the linear span $\Sigma = \langle p_1, \dots, p_{g-1} \rangle$ is $(g-2)$-dimensional,

b) 
$\Sigma \cap \Lambda = \emptyset$.

d) If $H_\Sigma \subset \P^{2g-1}$ is  the linear span of $\Lambda$ and $\Sigma$ (which , by b) is a hyperplane) then $\quad D:=H_\Sigma \cap S$ is a smooth curve.

We have a natural fibration
\begin{equation*}
\aligned
\beta \colon &U\longrightarrow\P V^{\dual}_{-}
\\
\{p_1, \dots,\, & p_{g-1}\}\mapsto H_\Sigma \cap \P V_{-}
\endaligned
\end{equation*}

Moreover, we set $C = f(D)$ and we observe that a point $\{p_1, \dots, p_{g-1}\} \in U \subset \Hilb^{g-1}(S)$ uniquely defines a divisor $\Delta = p_1 +\cdots + p_{g-1}$ on $D$ and therefore, since $g$ is odd, by (\ref{components ker norm}), the 
point $[\Delta - \iota \Delta]$ belongs to $P$. 
Thus one may define a morphism
\begin{equation} \label{phi}
\aligned
\phi \colon &U\longrightarrow P
\\
\{p_1,\dots,p_{g-1}&\} \mapsto[\Delta-\iota \Delta].
\endaligned
\end{equation}
This is how the rational map (\ref{hilbsch}) is defined.
We claim that 
the rational map $\phi$ is dominant.

By the way it is defined, the morphism $\phi$ defined in \eqref{phi} commutes with the two fibrations $\beta \colon U \to \P V^{\dual}_{-}$ and $p \colon P\to \P V^{\dual}_{-}$. Moreover, $\Hilb^{g-1}(S)$ and $P$ have the same dimension. Thus it suffices to show that the morphism
\begin{equation} \label{psi}
\aligned
\psi \colon &D_{g-1}\longrightarrow P
\\
\{p_1, \dots,& p_{g-1}\}\mapsto  [\Delta-\iota \Delta]
\endaligned
\end{equation}
where $D = p_1+ \cdots + p_{g-1}$, is dominant. In order to do this we show that the differential is an isomorphism at one point. Here, as usual, $D_{g-1}$ stands for the $(g-1)$-fold symmetric product of $D$. The morphism $\psi$ is the composition of the Abel-Jacobi morphism $u \colon D_{g-1} \to J(D)$ and the projection $1 - \iota \colon J(D) \to P$.
If $\omega_1, \dots, \omega_{2g-1}$ is a basis of $H^0(D, \omega_{D})$ and if the points $p_1, \dots, p_{g-1}$ are distinct, the rank of $u_{*}$ at the point $D = p_1+ \cdots + p_{g-1}$ is the rank of the Brill-Noether matrix $(\omega_i(p_j))$, $i=1, \dots, 2g - 1$, $j = 1, \dots, g - 1$. Let us now assume, as we may, that $\omega_1, \dots, \omega_{g}$ are $\iota$-invariant while $\omega_{g+1}, \dots, \omega_{2g-1}$ are $\iota$-anti-invariant. Then the rank of $\psi_{*}$ at $D$ is nothing but the rank of the $(g-1) \times (g-1)$ matrix $(\omega_i(p_j))$, $i = g + 1, \dots, 2g - 1$, $j = 1, \dots, g - 1$. But this matrix must be of maximal rank otherwise the linear span $\Sigma$ of the points $p_1, \dots, p_{g-1}$ would intersect the vertex $\Lambda = \P V_{+}$ contrary to the assumptions.
The existence of a dominant rational map from $\Hilb^{g-1}(S)$ implies that, if $\gamma: \wt P \to P$ is any desingularization of $P$, then
\begin{equation}\label{20_ineq}
h^{2,0}( \wt P) \leq h^{2,0}(\Hilb^{g-1}(S)) = 1,
\end{equation}
Let $\sigma$  be the holomorphic $2$-form defined on $P_{reg}$.
In order to prove that  (\ref{20_ineq}) is  an equality, 
it is enough  to show that the pull- back of $\sigma$ to $\gamma^{-1}(P_{reg})$ extends to $\wt P$. Let $\nu: \wh P\to P$ the normalization, and let $\wh \sigma$ be the pull-back of $\sigma$ to $\nu^{-1}(P_{reg})$. Using Proposition
\ref{codim_sing} and Hartog's theorem we  extend $\wh \sigma$ to $\wh P_{reg}$.
We reach the conclusion, using again
 Proposition
\ref{codim_sing} and a theorem of Flenner \cite{Flenner}
which guarantees that, given a normal variety $X$, a resolution of singularities $\alpha: \wt X\to X$,
and a holomorphic $2$-form $\omega$ on $X_{reg}$, then $\alpha^*(\omega)$ extends to
$\wt X$ as soon as 
$\codim_XX_{sing}\geq 4$. 

\begin{rem}{\rm As Voisin pointed out to us, we do not need to use Flenner's theorem to prove that the symplectic form $\sigma$ 
extends to any resolution  $\wt P$ of $P$. Indeed the isomorphism between $H^{2,0}(S)$ and  $H^{2,0}(M)$ is induced, up to a multiplicative constant,  by the correspondence $\Gamma\in \operatorname{CH}_2(S\times M)$, where $\Gamma$ is the second Chern character of a semi-universal family (see \cite{Kieran7}). This correspondence induces one in $\operatorname{CH}_2(S\times \wt P )$ giving a non-zero homomorphism  from $H^{2,0}(S)$ to $H^{2,0}(\wt P)$.
}
\end{rem}

\end{proof}


 \section{The discriminant }\label{discriminant}
 
 This section is devoted to the discussion of a numerological curiosity.
 
 The starting point of the story is Hitchin and Sawon's  study \cite{Hitchin-Sawon},  \cite{Sawon3} of the Rozanski-Witten invariant of a compact hyperk\"ahler manifold and the discovery of the following  remarkable formula 
linking the $L^2$- norm of the Riemann curvature
tensor $R$ of an irreducible  compact hyperk\"ahler manifold $X$
of real dimension $2n=4k$ with the characteristic number of $\sqrt{\wh A}[X]$ coming from the square root of the
${\wh A}$-polynomial:
$$
\frac{1}{(192\pi^2k)^k}\frac{||R||^{2k}}{(vol(X)^{k-1}}=\sqrt{\wh A}[X]
$$
 In  \cite{Sawon2},  \cite{Sawon4}, Sawon considers the case of a compact hyperk\"ahler manifold of complex dimension $n$ which is  a lagrangian fibration $\nu: X\to \PP^n$  by principally polarized abelian
varieties and looks at the {\it discriminant } $\Delta\subset \PP^n$
 (i.e. the set of points where $\nu$ is not a smooth morphism). The fibration should have {\it good singular fibers}  meaning that  the generic singular fibre $X_t$ for $t\in\Delta_{sm}=\Delta\smallsetminus\Delta_{sing}$ is obtained by gluing
together the zero and infinity sections of a $\PP^1$-bundle over a principally polarized
abelian variety of dimension $n-1$, i.e. a {\it rank one degeneration of an abelian variety}.  He then proves another remarkable formula:
 $$
 \deg \Delta=24\left(n!\sqrt{\wh A}[X]\right)^{\frac{1}{n}}
 $$
which makes the degree of the discriminant into a deformation invariant (this formula can be generalized to the case of non-principal polarizations).
He then proceeds to compute this invariant in the
two classical cases considered by Beauville. The first case is the Beauville-Mukai integrable system $X_n=J(|\Gamma|)\to \PP^n$ where $|\Gamma|$ is an $n$-dimensional linear system on K3 surface.
$$
 \deg \Delta_{X_n}=6(n+3)
 $$
 Using Lefschetz pencils this  is an easy Euler-Poincar\`e characteristic computation.
The second case is the generalized Kummer variety $K_n$,
 introduced by Beauville, and here the degree of the discriminant is given by
 $$
 \deg \Delta_{K_n}=6(n+1)
 $$
We now come to our (singular) Prym lagrangian fibrations
$$
P_n\to \PP^n
$$
Here, as usual we start with an Enriques surface $T$ and its universal covering $f: S\to T$. We take an irreducible curve $C$  of genus $g=n+1$ on $T$ and let $D=f^{-1}({C})$, $h=g(D)$ and $v=(0, D, -h+1)$.  We let $H$ be a polarization, and consider the relative Prym variety $P_n=\Prym_{v,H, N}$ for some line bundle $N$. We also assume that
$|C|$ contains a Lefschetz pencil.  Under these hypotheses we have:

\begin{prop}
a) If $|C|$ is non-hyperelliptic then
$$
\deg \Delta_{P_n}=6(n+2).
$$
b) If $|C|=|ne_1+e_2|$ is hyperelliptic then
$$
\deg \Delta_{P_n}=6(n+3).
$$
\end{prop}

\begin{proof}
Part a) follows once we prove the claim that, under our assumptions, the discriminant locus of $\nu$ coincides with the discriminant locus $\Delta_{|C|}$ of the linear system $|C|$ (recall that since we are assuming the existence of a Lefschetz pencil, the discriminant locus of the linear system is reduced). Indeed, if this is the case and easy Euler characteristic calculation shows that $\deg \Delta_{|C|}=6(g+1)$, so that substituting $n=g-1$ we get our result.
As for the claim, it follows from Corollary \ref{rank one degeneration} that the fiber of $\nu$ over a point corresponding to an irreducible curve with one node is a rank one degeneration of an abelian variety, and from the fact that there are no reducible curves in codimension one.

\vskip 0.2 cm 
As far as part b) is concerned, from Section \ref{hyp} the relative Prym  $P_n$ is smooth and of $K3^{[n]}$-type.
The statement follows then directly from Sawon's result.
However,  it is interesting to compute, by geometrical means 
 the degree of the discriminant $\Delta_{P_n}$ in this case.   This is done in \cite{Sacca13}. Here we only give a sketch of this analysis. As proved in Proposition \ref{discriminant_hyp}, the discriminant locus in the hyperelliptic case is the union of four irreducible components. One must then describe the fibers of $P_n=\Prym_{v,H, N}\to |C|$ over the general point of each of these components. This is done under the assumption that $H$ is a general polarization. Over the general point of
 $\Delta_1$ and $\Delta_2$ the Prym variety consists in  closed chain of  four
 irreducible components $R_1,\dots, R_4$ each meeting the successive one (in a cyclic order) transversally along an abelian variety. Since  $\Delta_1$ and $\Delta_2$ are hyperplanes their contribution to the degree
 of the discriminant is equal to 8. Over the general point of
 $\Delta_3$  the Prym variety is the union of two
 irreducible components meeting  transversally along two abelian varieties.
 Since  the degree of $\Delta_3$ is $n-1$,  its contribution to the degree
 of the discriminant is equal to $2(n-1)$. Finally, over the 
 general point of
 $\Delta_4$  the Prym variety is a rank one degeneration of an abelian variety, therefore  the  contribution of $\Delta_4$ to the degree
 of the discriminant is equal to its degree i.e. to $4n+12$. 
 Summing up we get $6(n+3)$.

 \end{proof}


\section{Further remarks}\label{concl_remarks}

1) It should be remarked that, \emph{in the non-hyperelliptic case} and for sufficiently high value of the genus, the Prym varieties in the fibers of $P_{v,H}\to |C|$ are definitely not Jacobians. In fact to make sure that this is so, according to Mumford's Theorem in section 7 of \cite{MumfordPryms}, we only have to make sure that $C$ is neither trigonal nor a double cover of an elliptic curve. On the other hand Corollary 1 in the paper \cite{Knutsen-Lopez-09} of Knutsen and Lopez tells us that the gonality  of general member of $|C|$ is equal to 
$2\phi(|C|)$. One can choose the linear system $|C|$ to make this number as large as one wishes.

2) Singular points of moduli spaces of sheaves have been extensively studied by
Kaledin, Lehn, Sorger, and Zhang, among others \cite{Kaledin-Lehn-Sorger06}, \cite{Kaledin-Lehn07} \cite{Lehn-Sorger06} \cite{Zhang12}. These authors are,  by and large,  interested in the case of   sheaves of rank $\geq 1$. The way they carry out their analysis consists in two basic techniques (with the notation of Section \ref{kuranishi}):

a) Prove that the Kuranishi family is {\it formal} which implies  that $B=Q$, or else

b) Prove directly that there is a local isomorphism $M\supset U\cong Q/G$

This reduces the local study of $M$ to the study of quotients $Q/G$; these are far more transparent objects, which are often similar to the {\it quiver varieties} of Nakajima \cite{Nakajima98},  \cite{Nakajima99}.

In a forthcoming paper \cite{Arbarello-Sacca12}, the authors will study  the local structure of
moduli  of rank zero sheaves by settling  the question of formality in some cases
and by establishing that in some of these  cases the quotients $Q/G$ turn out to be {\it exactly} the quiver varieties of Nakajima.

In  the present paper, we  only needed to examine  Kuranishi families $B$ which are hypersurfaces in $\Ext^1(F,F)$. In these cases, relations (\ref{cono_B}) and (\ref{cono_M})
hold and, as we  saw in Section \ref{kuranishi}, the local analysis is  quite  straightforward.

3) When $\tau$ is regular, it is natural to look at the quotient $M/\tau$
and to ask if it admits a symplectic resolution. It does not. Indeed it is enough to check that $M_{reg}/\tau$ has  no symplectic resolution. By Lemma 2.11 in Kaledin's paper \cite{Kaledin06} a symplectic resolution  $Z\to M_{reg}/\tau$  would be semismall
and this can not be the case since  $M_{reg}/\tau$
 is  $\Q$-factorial and the codimension of its singular locus is equal to $2g\geq 4$. When $g=1$, the moduli space $M$ is a K3 surface
 and $\tau$ is a symplectic involution with 8 fixed points.


\appendix
\section{Curves on K3 and Enriques surfaces}
\label{appendix}
We keep the notation and the assumptions of Section \ref{notation}.
 We start with the following Lemma,

\begin{lem}\label{prim_irr} Let $T$ be a general Enriques surface, and let $C\subset T$ be an irreducible curve. If the class of $C$ is not divisible by two in $\NS(T)$, then $D=f^{-1}({C})$ is irreducible.
\end{lem}
\begin{proof}
This follows immediately from the fact that, by assumption, $\NS(S)= f^* \NS(T)$.
\end{proof}

Let $A$ and $B$ be two effective classes on $T$, or on $S$. By the Hodge index theorem it follows that (cf. \cite{Knutsen-Lopez-07}), if $A^2,B^2\ge 0$, then $A\cdot B \ge 0$. Moreover, $A\cdot B =0$, if and only if, $\Z A=\Z B$ in  $\NS(T)$ and $A^2=B^2=0$.
By the Nakai-Moishezon-Kleiman criterion and the Hodge index theorem, it follows that if $S$ and $T$ are unodal, then 
\[
\Amp(T)=\mathcal Q_T^+,\,\,\, \text{ and } \Amp(S)=\mathcal Q_S^+,
\]
where $\mathcal  Q_T^+$ (and $\mathcal  Q_S^+$) is the connected component of the cone of classes with positive self intersection in $\NS(T)$, and $\NS(S)$ respectively, containing one ample class. Moreover, for both surfaces the cone of effective curves is equal to the closures of $ Q_T^+$ and $ Q_S^+$ in $\NS(T)$ and $\NS(S)$ respectively.

Unless otherwise specified, we will denote by $e$ or by $e_i$ primitive elliptic curves on $T$. Notice that $e'$ and $e_i'$ are also primitive elliptic curves on $T$. The curves $e$ and $e'$ are called the half-fibers of the elliptic pencil $|2e|=|2e'|$. The Kodaira classification of the singular fibers of an elliptic fibration implies that $e$ and $e'$ are either smooth or isomorphic to a closed chain of $b \ge 0$ smooth rational curves. It follows that if $T$ is general, then this forces $e$ and $e'$ to be smooth.
We will denote by $E$ and $E_i$, respectively,  their preimages in $S$.

The following two definitions  play an important role in the whole paper.
\begin{defin} [\cite{Bombieri73}] Let $m$ be a positive integer.
An effective divisor $C \subset T$ is said to be $m$-connected, if for every decomposition $C=C_1+C_2$ into the sum of two effective divisors we have $C_1 \cdot C_2 \ge m$. A linear system $|L|$ is said to be $m$-connected if all its members are $m$-connected.
\end{defin}

\begin{defin}[\cite{Saint-Donat}] \label{hyp_syst}A linear system $|L|$ on a K3 or an Enriques surface is said to be hyperelliptic if $L^2=2$ or if the associated morphism $\varphi_L$  is of degree $2$ onto a rational normal scroll of degree $n-1$ in $\P^n$.
\end{defin}

We now proceed to state a characterization of hyperelliptic linear systems on K3 or Enriques surfaces. We start with the following proposition (cf. Proposition 4.5.1 of \cite{Cossec-Dolgachev}).

\begin{prop}[\cite{Cossec-Dolgachev}] \label{hyperelliptic enriques}
Let $T$ be an Enriques surface, and let $C\subset T$ be an irreducible curve with $C^2 \ge 2$. The following are equivalent,
\begin{enumerate}
\item $|C|$ is a hyperelliptic curve;
\item $|C|$ has base points;
\item There exists a primitive elliptic curve $e$ such that $C\cdot e =1$.
\end{enumerate}
\end{prop}

By Corollary 4.5.1 of \cite{Cossec-Dolgachev} it  follows that the general member of a hyperelliptic linear system is a smooth hyperelliptic curve. The following proposition is a collection of results from \cite{Saint-Donat}.

\begin{prop} \label{hyperelliptic k3}
Let $S$ be a K3 surface and let $D\subset S$ be an irreducible curve, with $D^2 \ge 4$. The following are equivalent,
\begin{enumerate}
\item $|D|$ is a hyperelliptic;
\item the general member of $|D|$ is a smooth hyperelliptic curve;
\item there exists an elliptic pencil $|E_1|$ such that $C\cdot E_1=2$.
\end{enumerate}
Suppose, moreover, that $S$ is unodal. If one of the above conditions is satisfied,  then there exist an  integer $n \ge1$, and a primitive elliptic curve $E_2$ such that
\[
D=nE_1+E_2, \,\,\, \text{with} \,\, E_1\cdot E_2=2.
\]
Moreover, the morphism
\be\label{phi_d}
\varphi_D: S \to R\subset \P^{2n+1},
\ee
is of degree two and maps $S$ onto a rational normal scroll $R$ of degree $2n$ in $ \P^{2n+1}$, and $R$ is isomorphic to a quadric surface whose two rulings are the images under $\varphi_D$ of the elliptic pencils $|E_1|$ and $|E_2|$.
\end{prop}

Thus, if $|C|$ is hyperelliptic,  so is $|D|$. In particular, if   $T$ is general, the class of any hyperelliptic curve $C$ on $T$ is of the form
\[
ne_1+e_2,
\]
with $n \ge 1$ and $e_1\cdot e_2=1$. 
And hence we see that a hyperelliptic linear system is $1$-connected. It turns out that also the converse holds

\begin{prop} \label{nice decomposition on T}
Let $T$ be a general Enriques surface and let $L$ be an effective line bundle on $T$, with $L^2 >0$. Then $|L|$
contains a member that is the union of two smooth curves meeting transversally in $\nu\ge 1$ points. Moreover, $\nu=1$ if and only if $|L|$ is hyperelliptic.
\end{prop}
\begin{proof}
This follows from Corollary 3.2.2 and Proposition 4.3.4 of \cite{Cossec-Dolgachev}, and the fact that since $T$ is general any linear system with positive self intersection is ample and has no base component.
\end{proof}

 \begin{cor} \label{one connected} Let $T$ be a general Enriques surface and let $L$ be a line bundle on $T$. The linear system $|L|$ is hyperelliptic if and only if it is $1$-connected.
 \end{cor}

 \begin{cor} \label{two_curves}
Let $T$ be a general Enriques surface, and let $C\subset T$ be an irreducible curve. Set $D=f^{-1}(C)$. Then, there is a member of $f^*|C| \subset |D|$ that is a union of two smooth  $\iota$-invariant curves $D_1$ and $D_2$ meeting in $2\nu$ points. Moreover, we  have  $\nu\ge 2$, unless $|D|$ is hyperelliptic, in which case $\nu=1$.
\end{cor}
\begin{proof}
Keeping the notation in the proof of the previous proposition, it is enough to check that
the curves $D_1=f^{-1}(C_1)$ and $D_2=f^{-1}(C_2)$ are connected. This, however, follows from the fact that in Proposition \ref{nice decomposition on T} one can choose, for $i=1,2$, the curve $C_i$ such that either $C_i^2>0$ or $C_i$ is a primitive elliptic curve.
\end{proof}

Finally, we have the following proposition.

\begin{prop}\label{red_cod_one}
Let $X$ be a general Enriques surface or a K3 surface covering a general Enriques surface, and let $L$ be a line bundle on $X$ with $L^2>0$. Then $|L|$ has reducible members in codimension one if and only if $|L|$ is hyperelliptic.\\
\end{prop}

\begin{proof}
We will prove the proposition when $X$ is an Enriques surface, the proof for K3 surfaces is analogous. Since the irregularity of the surface is zero,  any component of the discriminant parameterizing reducible curve is birational to a product of linear systems. Consider a reducible member $C_1+C_2$, and set $\nu=C_1\cdot C_2$. For $i=1,2$, letting $g_i$ be the arithmetic genus of $C_i$, we have
\[
g=g_1+g_2+\nu-1.
\]
First assume that if, for some $i=1,2$, the genus $g_i$ is equal to one the corresponding curve $C_i$ is a primitive elliptic curve.
 Then $\dim|C_i|=g_i-1$, $\dim |C|= \dim |C_1|+ \dim |C_2|+\nu$, and $\codim (|C_1|\times |C_2|, |C|)=\nu$. From Corollary  \ref{one connected} it follows that $\nu \ge 1$ and that $\nu=1$ if and only if $C$ is hyperelliptic.

Next, consider the case $C_1\in |s e_1|=\P^{\lfloor \ff{s}{2}\rfloor}$ for some primitive elliptic curve $e_1$ and some integer $s \ge 2$. Then $\dim |C|= g_2-1+\nu$. Thus,  if $g_2 \ge 2$, we have $\dim |C_1|\times |C_2|=\lfloor \ff{s}{2}\rfloor+g_2-1$, while $\dim |C_1|\times |C_2|=\lfloor \ff{s}{2}\rfloor+\lfloor \ff{t}{2}\rfloor$, if $C_2=te_2$ with $t \ge 1$. It follows that 
\[
\codim (|C_1|\times |C_2|, |C|)=\left\{\begin{aligned} & \nu-\lfloor \ff{s}{2}\rfloor, && \text{if} \,\, g_2\ge 2\\ &\nu-\lfloor \ff{s}{2}\rfloor-\lfloor \ff{t}{2}\rfloor, && \,\text{if} \,\, g_2=1. \end{aligned}   \right.
\]
In the first case, since $\nu=s\nu'$, with $\nu' \ge 1$ we are done, unless $s=2$ and $\nu'=1$. However if $s=2$ and $\nu'=1$, the   curve $C_2$ is  hyperelliptic  of the form $\nu e_1+e_2$, with $e_1\cdot e_2=1$, and hence $L={\mc O}_T((\nu+s)e_1+e_2)$ is hyperelliptic. 

As for the second case, we can set $\nu=st\nu'$ and thus we are done unless $\nu'=1$, $s=2$ and $t=1$. This means that $C_2=e_2$, with $e_1\cdot e_2=1$ and, again, $L=2e_1+e_2$ is hyperelliptic.
\end{proof}

\begin{prop}\label{discriminant_hyp}
Let $T$ be an Enriques surface, and let $|C|=|ne_1+e_2|$ be a genus $g=n+1\ge 3$ hyperelliptic linear system on $T$. The discriminant locus $\Delta \subset |C|$ is the union of four irreducible components $\Delta_1, \Delta_2, \Delta_3, \Delta_4$. The first two are hyperplanes and their general point parametrizes curves that are the union of two smooth curves meeting transversally in one point. The third component, whose general point parametrizes curves that are union of two smooth curves meeting transversally in two points, is of degree $n-1$. The forth component is of degree $4n+12$ and it parametrizes singular but irreducible curves. Moreover, its general point  corresponds to an irreducible  curve with a single node.
\end{prop}
\begin{proof}

First of all,  via a straightforward   Euler characteristic count, one  shows that the degree of $\Delta$ is equal to $6(n+2)$.
It is clear that the two hyperplanes,
\[
\Delta_1:=[e_1]\times |(n-1)e_1+e_2|,\quad \text{ and } \quad \Delta_2:=[e_1']\times |(n-1)e_1+e_2'|,
\]
constitute two components of the discriminant locus, and also that they parametrizes curves of the form $e_1\cup \Gamma$ with $\Gamma \in |(n-1)e_1+e_2|$ (resp. $e_1'\cup \Gamma$ with $\Gamma \in |(n-1)e_1+e'_2|$.  Secondly,  the natural map
\[
\phi : \P^1 \times \P^{n-2}=|2e_1| \times |(n-2)e_1 \times e_2| \to \Delta \subset \P^n.
\]
is just the composition of the Segre embedding $ \P^1 \times \P^{n-2} \to \P^{2n-3}$ with the projection $ \P^{2n-3} \to \P^n$ induced by the natural map
\[
H^0(T, \O(2e_1)) \otimes H^0(T, \O((n-2) e_1+e_2)) \to H^0(T, \O(ne_1+e_2)).
\]
It follows that,
\[
\P^1 \times \P^{n-2} \to \phi(\P^1 \times \P^{n-2}),
\]
is finite and generically one to one, and that $\Delta_3:=\phi(\P^1 \times \P^{n-2})$ is a degree $n-1$ component of $\Delta$ whose general point parametrizes curves as in the statement of the theorem.
We just need to prove that the remaining part $\Delta_4$ of the discriminant is irreducible. Recall that by definition the rational map $\phi_{|C|}$ associated to the linear system maps $T$ generically $2:1$ onto a degree $n-1$ rational surface $R \subset \P^{n}$, and contracts the curves $e_1$ and $e_1'$.
A curve in $|C|$ is singular in the following three cases.
If it covers a singular (hence reducible) hyperplane section of $R$,  if it covers a smooth curve that is tangent to the ramification curve, or else if its image contains one of the two points $P$ and $Q$ to which either $e_1$ or $e_1'$ is contracted. The ramification  curve is described in Theorem 4.5.2 of \cite{Cossec-Dolgachev}.
It consists of the union of two lines $\ell_1$ and $\ell_2$ belonging  the ruling that is the image of $|2e_1|$ and of an irreducible curve $B \subset R$. The irreducible curve $B$ has two tacnodes in $P$ and $Q$ and is otherwise non singular (we are in the un-nodal case). 
The preimages, under $\phi_{|C|}$, of the lines  through $P$ and $Q$ are the curves in $\Delta_1$ and $\Delta_2$, respectively. 
The set of hyperplane sections of $R$ passing through $\ell_1$ or $\ell_2$ form a set of codimension $2$, whereas the set of hyperplane sections of $R$ that are tangent to $B$ form an irreducible divisor $\Delta_4 \subset \P^n$.
Moreover, the general hyperplane in this component is tangent to $B$ in just one point and thus corresponds to curves in $|C|$ with only one node and no other singularity.
\end{proof}

 \phantom{\cite{OGrady03} \cite{OGrady99}}

\bibliographystyle{abbrv}
\bibliography{bibliografia}

\end{document}